\documentclass[11pt,reqno]{amsart}
\usepackage{amsmath,amsfonts,amssymb,amsthm,amscd,comment,euscript}
\usepackage[all]{xy}
\usepackage{graphicx}
\usepackage{mathptmx}
\usepackage{enumerate} 
\usepackage[colorlinks=true]{hyperref} 
\usepackage[usenames,dvipsnames]{xcolor}
\usepackage{enumitem}

\xymatrixcolsep{1.9pc}                          
\xymatrixrowsep{1.9pc}
\newdir{ >}{{}*!/-5\pt/\dir{>}}                  

 \calclayout

\swapnumbers

\theoremstyle{plain}
\newtheorem{lem}{Lemma}[section]
\newtheorem{cor}[lem]{Corollary}
\newtheorem{prop}[lem]{Proposition}
\newtheorem{thm}[lem]{Theorem}

\theoremstyle{definition}

\newtheorem{rem}[lem]{Remark}
\newtheorem{dfn}[lem]{Definition}

\renewcommand{\phi}{\varphi}
\renewcommand{\leq}{\leqslant}
\renewcommand{\geq}{\geqslant}
\renewcommand{\epsilon}{\varepsilon}

\renewcommand{\kappa}{\varkappa}

\DeclareMathOperator{\spec}{Spec}
 \DeclareMathOperator{\cyl}{cyl}
 
\DeclareMathOperator{\diag}{diag}

 \DeclareMathOperator{\mot}{mot}

\DeclareMathOperator{\Hom}{Hom} 
 \DeclareMathOperator{\id}{id}
\DeclareMathOperator{\Fr}{Fr} 
 \DeclareMathOperator{\Alg}{Alg}
\DeclareMathOperator{\colim}{colim}
\DeclareMathOperator{\corr}{Corr} 
 
 \DeclareMathOperator{\kr}{Ker}
 
\DeclareMathOperator{\Ho}{Ho} 
\DeclareMathOperator{\nis}{\mathsf{nis}}

 \DeclareMathOperator{\Mod}{Mod}
 \DeclareMathOperator{\Ob}{Ob}

\newcommand{\bl}[1]{\buildrel #1\over}
\newcommand{\cc}{\mathcal}
\newcommand{\bb}{\mathbb}

\newcommand{\op}{{\textrm{\rm op}}}

\newcommand{\wt}{\widetilde}

\newcommand{\shnis}{SH^{\nis}_{S^{1}}}
\newcommand{\gmp}{\bb G}

\newcommand{\uhom}{\underline{\Hom}}

\newcommand{\M}{\mathcal{M}}

\newcommand{\eff}{\mathsf{eff}}

\newcommand{\pt}{\mathsf{pt}}

\begin{document}

\footskip30pt


\title{Correspondences and stable homotopy theory}
\author{Grigory Garkusha}
\address{Department of Mathematics, Swansea University, Fabian Way, Swansea SA1 8EN, UK}
\email{g.garkusha@swansea.ac.uk}

\begin{abstract}
A general method of producing correspondences and spectral categories
out of symmetric ring objects in general categories is given. As an application,
stable homotopy theory of spectra $SH$ is recovered from modules over a commutative symmetric ring
spectrum defined in terms of framed correspondences over an algebraically closed field. Another application 
recovers stable motivic homotopy theory $SH(k)$ from spectral modules over associated spectral categories.
\end{abstract}

\thanks{Supported by EPSRC grant EP/W012030/1}


\maketitle
\thispagestyle{empty}
\pagestyle{plain}

\tableofcontents

\section{Introduction}\label{introduction}

Algebraic Kasparov $K$-theory is stable homotopy theory of non-unital $k$-algebras $\Alg_k$~\cite{GG1,GG2}.
In detail, we start with the category $U_\bullet(\Alg_k)$ of pointed simplicial functors from $\Alg_k$ to pointed simplicial sets,
where each algebra $A\in\Alg_k$ is regarded as the representable object $rA=\Hom_{\Alg_k}(A,-)$. The category $U_\bullet(\Alg_k)$ 
comes equipped with a motivic model structure. Let $S^1$ be the standard simplicial circle. Stabilization of $U_\bullet(\Alg_k)$ in the $S^1$-direction
leads to a stable motivic model category $Sp_{S^1}(\Alg_k)$ of $S^1$-spectra in $U_\bullet(\Alg_k)$.
The $S^1$-suspension spectrum $\Sigma^\infty_{S^1} rA$ of an algebra $A$ is computed as the fibrant
spectrum $\bb K(A,-)$~\cite{GG2}, where $\bb K(A,B)$ is the algebraic Kasparov $KK$-theory spectrum of $A,B\in\Alg_k$
defined in~\cite{GG1}.
A key non-unital homomorphism involved in the computation is $\sigma_A:JA\to\Omega A$, where $JA=\kr(TA\to A)$ with
$TA=A\oplus A^{\otimes 2}\oplus\cdots$ the algebraic tensor algebra and 
$\Omega A=(x^2-x)A[x]$. The morphism $r(\sigma_A)$ is a motivic equivalence in $U_\bullet(\Alg_k)$. 

In stable motivic homotopy theory the suspension
$\bb P^1$-spectrum $\Sigma^\infty_{\bb P^1} X_+$ of a smooth
algebraic variety $X\in\mathrm{Sm}/k$ is computed in~\cite[Theorem~4.1]{GP1} as 
a (positively) fibrant spectrum $M_{\bb P^{\wedge 1}}(X)_f$. Locally in the Nisnevich topology it is equal to
the $\bb P^1$-spectrum $(\Fr(\Delta^\bullet_k\times-,X),\Fr(\Delta^\bullet_k\times-,X\times T),\ldots)$,
where $T=\bb A^1/(\bb A^1-\{0\})$ and $\Fr(-,X)$ is the sheaf of stable framed correspondences
introduced by Voevodsky~\cite{Voe2} in 2001.
A key morphism involved in the computation is the canonical motivic equivalence 
$\sigma_X:X_+\wedge\bb P^{\wedge 1}\to X_+\wedge T$ in the category of pointed motivic spaces $\cc M$. Here $\bb P^{\wedge 1}$
is the pointed projective line $(\bb P^{1},\infty)$. 

The computations of $\Sigma_{S^1}^\infty rA$ and $\Sigma^\infty_{\bb P^1} X_+$ share lots of common properties~\cite{GG17}.
Inspired by these computations, two categorical constructions are introduced in this paper.
The first one produces correspondences associated with 
two objects $P,T\in\cc C$ and
ring objects of the category of symmetric sequences $\cc C^\Sigma$, where $\cc C$ is
a symmetric monoidal category with finite colimits and zero object 0. The correspondences are constructed 
between objects
of an arbitrary full subcategory $\cc B$ of $\cc C$ closed under monoidal product. See Theorem~\ref{capitals} for details.
After Voevodsky, correspondences play a prominent role in motivic homotopy theory. In particular, they are necessary
for computing motivic homotopy types as well as for producing triangulated categories of motives.
For example, Voevodsky's fundamental graded category of framed correspondences $\Fr_*(k)$ is recovered from 
Theorem~\ref{capitals} if we take $\cc C=\cc M$, $\cc B=\{X_+\mid X\in\mathrm{Sm}/k\}$, $P=(\bb P^1,\infty)$, $T=\bb A^1/(\bb A^1-\{0\})$, and
the commutative ring object $(S^0,T,T^{2},\ldots)$ in $\cc M^\Sigma$. 
Next, if $\sigma:P\to T$ is a morphism in $\cc C$ then the second categorical 
construction produces spectral categories, i.e. categories
enriched over symmetric $S^1$-spectra $Sp_{S^1}^\Sigma$, which are used for applications mentioned below. 
See Theorem~\ref{tampa} for details.
Spectral categories are of great utility in classical and equivariant stable homotopy theory (see, e.g.,~\cite{GM,SS1})
as well as in constructing triangulated categories of $K$-motives~\cite{GP12,GP14}.

The spectral categories and symmetric spectra constructed in this paper lead to the following applications.
We first introduce the stable homotopy category $SH_k$ over an arbitrary field $k$ in Section~\ref{galois}. It is defined
as the homotopy category of $\mathsf{S}_k$-modules,
where $\mathsf{S}_k$ is a commutative symmetric ring spectrum defined over $k$.
Then one reconstructs in Theorem~\ref{recover} the
stable homotopy theory of $S^1$-spectra $SH$ as $SH_k$ if $k$ is algebraically closed 
(we need to invert the exponential characteristic). Moreover,
this reconstruction is given by a functor taking a symmetric $S^1$-spectrum $N$ to its 
symmetric framed motive ${M}^\Sigma_{fr}(N)$ introduced in this paper (see Definition~\ref{sigmafrmot}).
Another application gives yet another genuinely local model of stable motivic homotopy theory
$SH(k)$ (in addition to~\cite{GP5}) and, more generally, a local model for the category of $E$-modules
in $SH(k)$, where $E$ is a symmetric Thom ring spectrum. 
See Theorem~\ref{spectralmore} and Corollary~\ref{spectralmorecor1} for details.
For the latter result, we apply Theorem~\ref{tampa} to produce a spectral category $\cc O^E_{\Delta}$
using data as above: $\cc C=\cc M$, $\cc B=\{X_+\mid X\in\mathrm{Sm}/k\}$, $P=(\bb P^1,\infty)$, $T=\bb A^1/(\bb A^1-\{0\})$.
We also use the enriched motivic homotopy theory 
of motivic spectral categories developed in~\cite{GP12,GP14}.
The reader will also find reconstruction theorems for $E$-modules of $SH(k)$
in terms of $\infty$-categories of ``tangentially framed corrrespondences" in~\cite{EHKSY,EHKSY1}.
The approach presented in Section~\ref{enriched}
is combinatorial in the sense that it is based on explicit spectral categories produced by Theorem~\ref{tampa} and
modules over them defined in terms of original Voevodsky's framed correspondences~\cite{Voe2}.
This approach also produces triangulated categories of $E$-framed motives out of spectral 
categories of Theorem~\ref{tampa}. They are constructed in a similar fashion as the classical 
Voevodsky category of motives or the category of $K$-motives in the sense of~\cite{GP12,GP14}.

The author also expects further applications of (spectral) categories of correspondences, constructed in
this paper for quite general categories, in classical and algebraic Kasparov $K$-theory as well as in 
non-commutative algebraic geometry.
This will be the material of subsequent papers.
In this paper he has concentrated on applications in classical and motivic stable homotopy theory.

The author thanks the anonymous referee for helpful comments.


\subsubsection*{Notation.}
Throughout the paper we employ the following notation.\label{ntn}
\vspace{0.08in}

\begin{tabular}{l|l}
$k$ and $\pt$ & a field of exponential characteristic $e$ and $\mathrm{Spec}(k)$\\
$\mathrm{Sm}/k$ & the category of smooth separated schemes of finite type \\
$\Fr_0(k)$ or $\mathrm{Sm}/k_{+}$ & the category of framed correspondences of level zero\\
$\mathrm{Shv}_{\bullet}(\mathrm{Sm}/k)$ & the closed symmetric monoidal category of pointed Nisnevich sheaves \\
$\M=\Delta^{\op}\mathrm{Shv}_{\bullet}(\mathrm{Sm}/k)$ & the category of pointed motivic spaces,\\
{}& a.k.a. the category of~pointed simplicial  Nisnevich sheaves \\
${\bf S}_{\bullet}$ & the category of pointed simplicial sets
\end{tabular}
\vspace{0.05in}
\noindent 


\section{Graded symmetric sequences}\label{graded}

Let $(\cc C,\wedge,\vee,S)$ be a symmetric monoidal category with finite coproducts, unit object $S$ and zero object 0.
We assume that a canonical morphism
   $$v:\bigvee_{i\in I}(A_i\wedge B)\to(\bigvee_{i\in I}A_i)\wedge B$$
is an isomorphism for any finite set $I$ and $A_i,C\in\cc C$\footnote{Such a category $\cc C$ is also known as
a distributive symmetric monoidal category.}. In particular, if $I=\emptyset$ then $0\wedge B\cong B\wedge 0\cong 0$.

In what follows we shall also assume that $\cc C$ has finite colimits.
By~\cite[Section~7]{H} the category of symmetric sequences
$\cc C^\Sigma$ is symmetric monoidal with
   $$(X\wedge Y)_n=\bigvee_{p+q=n}\Sigma_n\times_{\Sigma_p\times\Sigma_q}X_p\wedge Y_q.$$
The symmetric sequence $(S,0,0,\ldots)$ is a monoidal unit of $\cc C^\Sigma$.
This notation needs some explanation (we follow~\cite[Section~7]{H}).
Given a finite set $\Gamma$ and an object $A\in\cc C$, $\Gamma\times A$ is
the coproduct of $|\Gamma|$ copies of $A$. If $\Gamma$ is a group, then $\Gamma\times A$ has an obvious left
$\Gamma$-action; $\Gamma\times A$ is the free $\Gamma$-object on $A$.
Note that a $\Gamma$-action on $A$ is then equivalent to a map $\Gamma\times A\to A$
satisfying the usual unit and associativity conditions. Also, if $\Gamma$ admits a right action by a
group $\Gamma'$, and $A$ is a left $\Gamma'$-object, then we can form $\Gamma\times_{\Gamma'} A$
as the colimit of the $\Gamma'$-action on $\Gamma\times A$,
where $\alpha\in\Gamma'$ takes the copy of $A$ corresponding to
$\beta\in\Gamma$ to the copy of $A$ corresponding to $\beta\alpha^{-1}$ by the action of $\alpha$.

Given two maps $f:X\to X'$ and $g:Y\to Y'$ in $\cc C^\Sigma$, the $\Sigma_p\times\Sigma_q$-equivariant maps
$f_p\wedge g_q:X_p\wedge Y_q\to X_p'\wedge Y_q'$ yield the 
{\it tensor product morphism\/} $f\wedge g:X\wedge Y\to X'\wedge Y'$ in $\cc C^\Sigma$.

The twist isomorphism $twist:X\wedge Y\to Y\wedge X$ for $X,Y\in\cc C^\Sigma$ is the
natural map taking the summand $(\alpha,X_p\wedge Y_q)$ to the summand $(\alpha\chi_{q,p},Y_q\wedge X_p)$
for $\alpha\in\Sigma_{p+q}$, where $\chi_{q,p}\in\Sigma_{p+q}$ is the $(q,p)$-shuffle
given by $\chi_{q,p}(i) = i+p$ for $1\leq i \leq q$ and $\chi_{q,p}(i) = i-q$ for $q < i \leq p + q$.
It is worth noting that the map defined without the shuffle permutation is not a map of symmetric sequences.

\begin{dfn}[Ring objects]
In what follows we shall refer to monoid objects in $\cc C^\Sigma$ as
{\it ring objects}. There is a standard description of a ring object  $E$ in $\cc C^\Sigma$ which we need later:

\begin{itemize}
\item[$\diamond$] a sequence of objects $E_n\in\cc C$ for $n\geq 0$;
\item[$\diamond$] a left action of the symmetric group $\Sigma_n$ on $E_n$ for each $n\geq 0$;
\item[$\diamond$] $\Sigma_n\times\Sigma_m$-equivariant {\it multiplication maps}
   $$\mu_{n,m}:E_n\wedge E_m\to E_{n+m}$$
for $n, m\geq 0$, and
\item[$\diamond$] a {\it unit map} $\iota_0:S\to E_0$.
\end{itemize}
This data is subject to the following conditions:

(Associativity) The square
   $$\xymatrix{E_n\wedge E_m\wedge E_p\ar[d]_{\mu_{n,m}\wedge\id}\ar[r]^{\id\wedge\mu_{m,p}}&E_n\wedge E_{m+p}\ar[d]^{\mu_{n,m+p}}\\
                       E_{n+m}\wedge E_{p}\ar[r]_{\mu_{n+m,p}}&E_{n+m+p}}$$
commutes for all $n, m, p\geq 0$.

(Unit) The two composites
   $$E_n\cong E_n\wedge S\xrightarrow{E_n\wedge\iota_0}E_n\wedge E_0\xrightarrow{\mu_{n,0}}E_n$$
   $$E_n\cong S\wedge E_n\xrightarrow{\iota_0\wedge E_n}E_0\wedge E_n\xrightarrow{\mu_{0,n}}E_n$$
are the identity for all $n\geq 0$.

A morphism $f:E\to E'$ of ring objects consists of $\Sigma_n$-equivariant maps
$f_n:E_n\to E'_n$ for $n\geq 0$, which are compatible with the multiplication and unit maps in the sense that
$f_{n+m}\circ\mu_{n,m}=\mu_{n,m}\circ(f_n\wedge f_m)$ for all $n,m\geq 0$, and $f_0\circ\iota_0=\iota_0$.

A ring object $E$ is commutative if the square
   $$\xymatrix{E_n\wedge E_m\ar[d]_{\mu_{n,m}}\ar[r]^{twist}&E_m\wedge E_{n}\ar[d]^{\mu_{m,n}}\\
                       E_{n+m}\ar[r]_{\chi_{n,m}}&E_{m+n}}$$
commutes for all $n,m\geq 0$.
\end{dfn}

\begin{dfn}[Modules]\label{emodule}
A {\it right module $M$\/} over a ring object $E\in\cc C^\Sigma$ is defined in a standard way. There is an equivalent
definition which we need later:
\begin{itemize}
\item[$\diamond$] a sequence of objects $M_n\in\cc C$ for $n\geq 0$
\item[$\diamond$] a left action of the symmetric group $\Sigma_n$ on $M_n$ for each $n\geq 0$, and
\item[$\diamond$] $\Sigma_n\times\Sigma_m$-equivariant action maps $\alpha_{n,m}:M_n\wedge E_m\to M_{n+m}$ for $n,m\geq 0$.
\end{itemize}
The action maps have to be associative and unital in the sense that the following diagrams commute
   $$\xymatrix{M_n\wedge E_m\wedge E_p\ar[r]^{M_n\wedge\mu_{m,p}}\ar[d]_{\alpha_{n,m}\wedge E_p}&M_n\wedge E_{m+p}\ar[d]^{\alpha_{n,m+p}}
                       &&&M_n\cong M_n\wedge S\ar[dr]_{\id_{M_n}}\ar[r]^(.57){M_n\wedge\iota_0}&M_n\wedge E_0\ar[d]^{\alpha_{n,0}}\\
                       M_{n+m}\wedge E_p\ar[r]_{\alpha_{n+m,p}}&M_{n+m+p}
                       &&&&M_n}$$
for all $n,m,p\geq 0$. A {\it morphism\/} $f:M\to N$ of right $E$-modules consists of
$\Sigma_n$-equivariant maps $f_n:M_n\to N_n$ for $n\geq 0$, which are compatible with the action
maps in the sense that $f_{n+m}\circ\alpha_{n,m}=\alpha_{n,m}\circ(f_n\wedge E_m$) for all
$n,m\geq 0$. We denote the category of right $E$-modules by $\Mod E$.
\end{dfn}

The following definition is motivated by the fundamental graded category of Voevodsky's framed correspondences,
in which the role of $\cc C$ is played by the category of pointed motivic spaces $\cc M$, $\cc B$ is given by pointed motivic spaces of the form
$X_+$ with $X\in\textrm{Sm}/k$, $P$ is the pointed projective line
$(\bb P^1,\infty)$, and $E=(\pt_+,T,T^2,\ldots)$ with $T=\bb A^1/(\bb A^1-\{0\})$.

In what follows we shall tacitly use iterated monoidal products and coherence.

\begin{dfn}
Suppose $\cc B$ is a full subcategory of $\cc C$ closed under
$\wedge$ and $P\in\Ob\cc C$. Let $E$ be a ring object of $\cc
C^\Sigma$. We define the set of {\it $(E,P)$-correspondences of
level $n$\/} between two objects $X,Y\in\cc B$ by
   $$\corr^E_n(X,Y):=\Hom_{\cc C}(X\wedge P^{\wedge n},Y\wedge E_n).$$
This set is pointed at the zeroth map. By definition, $\corr^E_0(X,Y):=\Hom_{\cc C}(X,Y\wedge E_0)$.
\end{dfn}

Define a pairing
   $$\phi_{X,Y,Z}:\corr^E_n(X,Y)\wedge\corr^E_m(Y,Z)\to\corr^E_{n+m}(X,Z)$$
by the rule: $\phi_{X,Y,Z}(f:X\wedge P^{\wedge n}\to Y\wedge E_n,g:Y\wedge P^{\wedge m}\to Z\wedge E_m)$
is given by the composition
\begin{gather*}
X\wedge P^{\wedge n}\wedge P^{\wedge m}\xrightarrow{f\wedge P^{\wedge m}}
Y\wedge E_n\wedge P^{\wedge m}
\xrightarrow{tw} Y\wedge P^{\wedge m}\wedge E_n\xrightarrow{g\wedge E_n} Z\wedge E_{m}\wedge E_n\xrightarrow{tw}\\
\to Z\wedge E_n\wedge E_m\xrightarrow{Z\wedge \mu_{n,m}}Z\wedge E_{n+m}.
\end{gather*}

\begin{thm}\label{capitals}
Let $E$ be a ring object in $\cc C^\Sigma$ and $\cc B$ is a full subcategory of $\cc C$ closed under monoidal product.
Then $\cc B$ can be enriched over the closed symmetric monoidal category of symmetric sequences of pointed sets $Sets_*^\Sigma$.
Namely, $Sets_*^\Sigma$-objects of morphisms are defined by
   \begin{equation}\label{skobki}
    (\corr^E_0(X,Y),\corr^E_1(X,Y),\corr^E_2(X,Y),\ldots),\quad X,Y\in\cc B.
   \end{equation}
Compositions are defined by
pairings $\phi_{X,Y,Z}$. The resulting $Sets_*^\Sigma$-category is denoted by $\corr^E_*(\cc B)$.
Moreover, if $E$ is a commutative ring object then $\corr^E_*(\cc B)$ is symmetric monoidal. The
monoidal product of $X,Y\in\Ob(\corr^E_*(\cc B))$ is defined as $X\wedge Y$, where $\wedge$ is
the monoidal product in $\cc B$.
\end{thm}

\begin{proof}
The left action of the symmetric group $\Sigma_n$ on $\corr^E_n(X,Y)=\Hom_{\cc
C}(X\wedge P^{\wedge n},Y\wedge E_n)$ for each $n\geq 0$ is given by
conjugation. In detail, for each $f:X\wedge P^{\wedge n}\to Y\wedge E_n$
and each $\tau\in\Sigma_n$ the morphism $\tau\cdot f$ is defined as the composition
   \begin{equation}\label{sigmaseq}
    X\wedge P^{\wedge n}\xrightarrow{X\wedge\tau^{-1}}X\wedge P^{\wedge n}\xrightarrow{f}Y\wedge E_n
       \xrightarrow{Y\wedge\tau}Y\wedge E_n.
   \end{equation}
With this definition each $\corr^E_n(X,Y)$ becomes a pointed $\Sigma_n$-set. Here $\Sigma_n$
acts on $P^{\wedge n}$ by permutations, using the commutativity and associativity isomorphisms.

Since the multiplication maps $\mu_{*,*}$ for $E$ are $\Sigma_n\times\Sigma_m$-equivariant and
the diagram\footnotesize
   $$\xymatrix{X\wedge P^{\wedge n}\wedge P^{\wedge m}
       &&&&Z\wedge E_n\wedge E_m\ar[r]^{Z\wedge\mu_{n,m}}&Z\wedge E_{n+m}\\
       X\wedge P^{\wedge n}\wedge P^{\wedge m}\ar[u]^{\id_X\wedge\alpha\wedge\beta}\ar[r]^{f\wedge P^{\wedge m}}
       &Y\wedge E_n\wedge P^{\wedge m}\ar[r]^{tw}
       &Y\wedge P^{\wedge m}\wedge E_n\ar[r]^{g\wedge E_n}
       &Z\wedge E_{m}\wedge E_n\ar[r]^{tw}
       &Z\wedge E_{n}\wedge E_m\ar[u]^{\id_Y\wedge\alpha\wedge\beta}\ar[r]^{Z\wedge\mu_{n,m}}
       &Z\wedge E_{n+m}\ar[u]_{\id_Y\wedge(\alpha\times\beta)}}$$
\normalsize is commutative for any $\alpha\in\Sigma_n,\beta\in\Sigma_m$, it follows that
the pairing $\phi_{X,Y,Z}$ is $\Sigma_n\times\Sigma_m$-equivariant.

If there is no likelihood of confusion, we shall simetimes write $(X\wedge P^{\wedge n},Y\wedge E_n)$ to denote
$\Hom_{\cc C}(X\wedge P^{\wedge n},Y\wedge E_n)$. The ``associativity square"\footnotesize
   $$\xymatrix{(X\wedge P^{\wedge n},Y\wedge E_n)\wedge(Y\wedge P^{\wedge m},Z\wedge E_m)
                      \wedge(Z\wedge P^{\wedge p},W\wedge E_p)\ar[r]^(.56){\id\wedge\mu_{m,p}}\ar[d]_{\mu_{n,m}\wedge\id}
                      &(X\wedge P^{\wedge n},Y\wedge E_n)\wedge(Y\wedge P^{\wedge m+p},W\wedge E_{m+p})\ar[d]^{\mu_{n,m+p}}\\
                      (X\wedge P^{\wedge n+m},Z\wedge E_{n+m})\wedge(Z\wedge P^{\wedge p},W\wedge E_{p})\ar[r]^(.55){\mu_{n+m,p}}
                      &(X\wedge P^{\wedge n+m+p},W\wedge E_{n+m+p})}$$
\normalsize is commutative for all $n, m, p\geq 0$ due to the associativity of the multiplication maps $\mu_{*,*}$ for $E$, and so
$\phi_{X,Y,Z}$ is an associative pairing.

The identity morphism is defined by
   $$u_X:X\bl{\rho^{-1}}\cong X\wedge S\xrightarrow{\id_X\wedge\iota_{0}}X\wedge E_0\in\corr_0^E(X,X),$$
where $\iota_0:S\to E_0$ is the unit map.
We see that $\cc B$ is enriched over $Sets_*^\Sigma$ by means of $(E,P)$-correspondences.

Suppose $E$ is a commutative ring object in $\cc C^\Sigma$.
If there is no likelihood of confusion, we shall simetimes write $[X,Y]$ to denote the $Sets_*^\Sigma$-object of morphisms~\eqref{skobki}.
To show that $\corr^E_*(\cc B)$ is a symmetric monoidal $Sets_*^\Sigma$-category, we need to define a $Sets_*^\Sigma$-functor
   $$\psi:\corr^E_*(\cc B)\wedge\corr^E_*(\cc B)\to\corr^E_*(\cc B),$$
where $\Ob(\corr^E_*(\cc B)\wedge\corr^E_*(\cc B))=\Ob\corr^E_*(\cc B)\times\Ob\corr^E_*(\cc B)$
and $[(X,Y),(X',Y')]=[X,X']\wedge[Y,Y']$ (see~\cite[p.~305]{Bor}). 
By definition, $\psi(X,Y)=X\wedge Y$ for all $(X,Y)\in\Ob(\corr^E_*(\cc B)\wedge\corr^E_*(\cc B))$.

Composition law in $\corr^E_*(\cc B)\wedge\corr^E_*(\cc B)$ is given by (see~\cite[p.~305]{Bor} for details)
   \begin{gather*}
   [(X,Y),(X',Y')]\wedge[(X',Y'),(X'',Y'')]=[(X,X')]\wedge[Y,Y']\wedge[(X',X'')]\wedge[Y',Y'']\xrightarrow{twist}\\
   [(X,X')]\wedge[(X',X'')]\wedge[Y,Y']\wedge[Y',Y'']\xrightarrow{\phi_{X,X',X''}\wedge\phi_{Y,Y',Y''}}[(X,X'')]\wedge[Y,Y'']=[(X,Y),(X'',Y'')].
   \end{gather*}
It sends each quadruple
$$(f:X\wedge P^{\wedge p}\to X'\wedge E_p,g:Y\wedge P^{\wedge q}\to Y'\wedge E_q)
\wedge(f':X'\wedge P^{\wedge s}\to X''\wedge E_s,g':Y'\wedge P^{\wedge t}\to Y''\wedge E_t)$$
to
$$(f:X\wedge P^{\wedge p}\to X'\wedge E_p,(\chi_{s,q},f':X'\wedge P^{\wedge s}\to X''\wedge E_s),
g:Y\wedge P^{\wedge q}\to Y'\wedge E_q,g':Y'\wedge P^{\wedge t}\to Y''\wedge E_t)$$
and then to the couple
   $$(\wt\chi_{s,q},f\circ f':X\wedge P^{\wedge p+s}\to X''\wedge E_{p+s},g\circ g':Y\wedge P^{\wedge q+t}\to Y''\wedge E_{q+t}),$$
where $\wt\chi_{s,q}\in\Sigma_{p+s+q+t}$ is the permutation $\id\times\chi_{s,q}\times\id\in\Sigma_{p}\times\Sigma_{s+q}\times\Sigma_{t}$.

Define
   $$\psi^{X,Y}_{X',Y'}:\corr^E_p(X,X')\wedge\corr^E_q(Y,Y')\to\corr^E_{p+q}(X\wedge Y,X'\wedge Y')$$
by sending $(f:X\wedge P^{\wedge p}\to X'\wedge E_p,g:Y\wedge P^{\wedge q}\to Y'\wedge E_q)$
to the composition
\begin{gather*}
X\wedge Y\wedge P^{\wedge p}\wedge P^{\wedge q}\xrightarrow{tw}
X\wedge P^{\wedge p}\wedge Y\wedge P^{\wedge q}\xrightarrow{f\wedge g}
X'\wedge E_p\wedge Y'\wedge E_q
\xrightarrow{tw} X'\wedge Y'\wedge E_p\wedge E_q\\
\xrightarrow{Y\wedge Y'\wedge\mu_{p,q}}X'\wedge Y'\wedge E_{p+q}.
\end{gather*}
The pairing $\psi^{X,Y}_{X',Y'}$ is plainly $\Sigma_p\times\Sigma_q$-equivariant.


We have a commutative diagram
   $$\xymatrix{E_p\wedge E_s\wedge E_q\wedge E_t\ar[d]_{\id\wedge\mu_{s,q}\wedge\id}\ar[rr]^{\id\wedge twist\wedge\id}
                      &&E_p\wedge E_q\wedge E_s\wedge E_t\ar[d]^{\id\wedge\mu_{q,s}\wedge\id}\ar[rr]^{\mu_{p,q}\wedge\mu_{s,t}}
                      &&E_{p+q}\wedge E_{s+t}\ar[d]^{\mu_{p+q,s+t}}\\
                      E_p\wedge E_{s+q}\wedge E_t\ar[d]_{\mu_{p,s+q}\wedge\id}\ar[rr]^{\id\wedge\chi_{s,q}\wedge\id}
                      &&E_p\wedge E_{q+s}\wedge E_t\ar[d]^{\mu_{p,q+s}\wedge\id}&&E_{p+q+s+t}\ar@{=}[ddll]\\
                      E_{p+s+q}\wedge E_t\ar[d]_{\mu_{p+s+q,t}}\ar[rr]^{\id\wedge\chi_{s,q}\wedge\id}
                      &&E_{p+q+s}\wedge E_t\ar[d]_{\mu_{p+q+s,t}}\\
                      E_{p+s+q+t}\ar[rr]_{\id\wedge\chi_{s,q}\wedge\id}&&E_{p+q+s+t}}$$
Commutativity of the top square follows from commutativity of the multiplication maps,
commutativity of the remaining squares follows from equivariancy of the multiplication maps.
It follows that
   $$\phi_{X\wedge Y,X'\wedge Y',X''\wedge Y''}(f\wedge g,f'\wedge g')=\phi_{X,X',X''}(f,f')\wedge\phi_{Y,Y',Y''}(g,g').$$\

For $X,Y\in\cc B$, the corresponding unit map of $\corr^E_*(\cc B)\wedge\corr^E_*(\cc B)$ is given by (see~\cite[p.~305]{Bor})
   $$S\xrightarrow{\rho^{-1}}S\wedge S\xrightarrow{u_X\wedge u_Y}[X,X]\wedge[Y,Y].$$
Clearly, $\psi^{X,Y}_{X,Y}$ takes it to the unit $u_{X\wedge Y}:S\to[X\wedge Y,X\wedge Y]$.

We see that $\psi:\corr^E_*(\cc B)\wedge\corr^E_*(\cc B)\to\corr^E_*(\cc B)$ is a $Sets_*^\Sigma$-functor.
It defines a symmetric monoidal $Sets_*^\Sigma$-category structure on $\corr^E_*(\cc B)$ with
the unit $S\in\Ob\corr^E_*(\cc B)$. $Sets_*^\Sigma$-natural associativity, symmetry
isomorphisms and two $Sets_*^\Sigma$-natural unit isomorphisms are inherited from the
same isomorphisms of the symmetric moinoidal category structure on $\cc B$. This completes the proof of the theorem.
\end{proof}

The proof of the theorem implies the following result.

\begin{cor}
Under the notation of Theorem~\ref{capitals}
any morphism of two ring objects $\gamma:E\to E'$ in $\cc C^\Sigma$ induces a morphism of $Sets_*^\Sigma$-categories
$\gamma_*:\corr^E_*(\cc B)\to\corr^{E'}_*(\cc B)$.
\end{cor}

Day's Theorem~\cite{Day} together with Theorem~\ref{capitals} also imply the following result.

\begin{cor}
Under the notation of Theorem~\ref{capitals}
if $E$ is a commutative ring object in $\cc C^\Sigma$ then the category of 
$\corr^E_*(\cc B)$-modules is a closed symmetric
monoidal category.
\end{cor}

\section{Symmetric spectra}\label{symm}

Let $\cc C$ be the category from the previous section and let $T$ be an object of $\cc C$.
Following~\cite[Section~7]{H}
consider the free commutative monoid $Sym(T)$ on the object $(0,T,0,0,\ldots)$
of $\cc C^\Sigma$ . Then $Sym(T)$ is the symmetric sequence $(S,T,T^{\wedge 2},T^{\wedge 3},\ldots)$,
where $\Sigma_n$ acts on $T^{\wedge n}$ by permutation, using the commutativity and associativity isomorphisms.

\begin{dfn}
(1) Following~\cite[Definition~7.2]{H} the {\it category of symmetric spectra $Sp^\Sigma(\cc C,T)$} is
the category of modules in $\cc C^\Sigma$ over the commutative monoid
$Sym(T)$ in $\cc C$. That is, a symmetric spectrum $X$ is a sequence of
$\Sigma_n$-objects $X_n\in\cc C$ and $\Sigma_n$-equivariant maps $X_n\wedge T\to X_{n+1}$,
such that the composite
   $$X_n\wedge T^{\wedge p}\to X_{n+1}\wedge T^{\wedge p-1}\to\cdots\to X_{n+p}$$
is $\Sigma_n\times\Sigma_p$-equivariant for all $n,p\geq 0$. A map of symmetric spectra is a collection of
$\Sigma_n$-equivariant maps $X_n\to Y_n$ compatible with the structure maps of $X$ and $Y$.

(2) A {\it symmetric ring $T$-spectrum\/} is a ring spectrum $E\in\cc C^\Sigma$ such that there is another unit map
$\iota_1:T\to E_1$ subject to the following condition:

(Centrality) The diagram
   $$\xymatrix{E_n\wedge T\ar[d]_{twist}\ar[r]^{E_n\wedge\iota_1}&E_n\wedge E_1\ar[r]^{\mu_{n,1}}&E_{n+1}\ar[d]^{\chi_{n,1}}\\
                       T\wedge E_n\ar[r]_{\iota_1\wedge E_n}& E_1\wedge E_n\ar[r]_{\mu_{1,n}}& E_{1+n}}$$
commutes for all $n\geq 0$.

$E$ is {\it commutative\/} if it is commutative as a ring object of $\cc C^\Sigma$.
A {\it morphism $f:E\to E'$} of symmetric ring spectra is a morphism of ring objects in $\cc C^\Sigma$
such that $f_1\iota_1=\iota_1'$.

(3) A {\it right module\/} $M$ over a symmetric ring $T$-spectrum $E$ is a symmetric right $T$-spectrum
which is also a right $E$-module in the sense of Definition~\ref{emodule}. We denote the category of right
$E$-modules by $\Mod E$ (its morphisms are morphisms of symmetric $T$-spectra satisfying
Definition~\ref{emodule}).
\end{dfn}

Because $Sym(T)$ is a commutative monoid, the category $Sp^\Sigma(\cc C,T)$ is a
symmetric monoidal category, with $Sym(T)$ itself as the unit.
We denote the monoidal structure by $X{{\bigwedge}} Y=X\wedge_{Sym(T)}Y$, where $X\wedge_{Sym(T)}Y$
is defined similarly to~\cite[p.~499]{SS} as the coequalizer, in $\cc C^\Sigma$, of the two maps
$\xymatrix{X\wedge Sym(T)\wedge Y \ar@<-.5ex>[r] \ar@<.5ex>[r] & X\wedge Y}$ induced by the actions of
$Sym(T)$ on $X$ and $Y$ respectively.

Given $X\in\cc C$ and a pointed set $(K,*)$, we shall write $X\wedge K$ to denote $\bigvee_{K\setminus *}X$.
Let $P\in\cc C$ and $\sigma:P\to T$ be a morphism in $\cc C$.
Denote by $T^n:=T\wedge\bl{n}\cdots\wedge T$ (respectively $P^{\wedge n}=P\wedge\bl{n}\cdots\wedge P$). 
This notation is inherited from the standard notation for pointed motivic spaces $T^n$
and $\bb P^{\wedge n}$ associated with pointed motivic spaces $T=\bb A^1/(\bb A^1-\{0\})$ and $(\bb P^{1},\infty)$,
where $\sigma:(\bb P^{1},\infty)\to T$ is the canonical motivic equivalence given by the level 1 framed correspondence
$(\{0\},\bb A^1,t)\in\Fr_1(\pt,\pt)$.

In what follows $S^1$ is the standard simplicial circle with $n$-simplicies being $n_+=\{0,1,\ldots,n\}$
and $S^n:=S^1\wedge\bl{n}\cdots\wedge S^1$.
Given
a symmetric right $T$-spectrum $E$, let $\mathbf{1}^E$ denote the
following $S^1$-spectrum in $\mathbf S_\bullet$:
   $$\mathbf{1}^E:=(\Hom_{\cc C}(S,E_0),\Hom_{\cc C}( P,E_1\wedge S^1),
       \Hom_{\cc C}(P^{\wedge 2},E_2\wedge S^2),\ldots).$$
Here each $E_n\wedge S^n$ is performed in every degree to produce a simplicial object in $\cc C$ and a simplicial Hom-set.
We also call $\mathbf{1}^E$ the {\it $PT$-spectrum of $E$}.

Each simplicial Hom-set is pointed at the zero morphism. Each structure map
   $$u_n:\Hom_{\cc C}(P^{\wedge n},E_n\wedge S^n)\wedge S^1\to\Hom_{\cc C}(P^{\wedge n+1},E_{n+1}\wedge S^{n+1})$$
coincides termwise with the natural morphisms
   $$\bigvee\Hom_{\cc C}(P^{\wedge n},E_n\wedge S^n)\to\Hom_{\cc C}(P^{\wedge n+1},\bigvee(E_{n+1}\wedge S^{n})),$$
where coproducts are indexed by non-basepoint elements of $S_n^1 = n_+=\{0,1,\ldots,n\}$. They take an element
$f:P^{\wedge n}\to E_n\wedge S^n$ of the $k$th summand to the composition
   $$P^{\wedge n+1}\xrightarrow{f\wedge\sigma} (E_n\wedge S^n)\wedge T\bl{v^{-1}}\cong (E_n\wedge T)\wedge S^n
       \to E_{n+1}\wedge S^n\bl{\iota_k}\hookrightarrow\bigvee E_{n+1}\wedge S^n.$$
Here $\iota_k$ is the inclusion into the $k$th summand.
If $E=Sym(T)$ then $\mathbf{1}^E$ will be denoted by $\mathbf{1}^\Sigma$.

Though the author was unable to find the following result in the literature, he does not pretend to originality.

\begin{thm}\label{symmsp}
Given a symmetric right $T$-spectrum $E\in Sp^\Sigma(\cc C,T)$, the following statements are true:
\begin{enumerate}
\item the spectrum $\mathbf{1}^E$ is a symmetric $S^1$-spectrum;
\item if $E$ is a (commutative) symmetric ring $T$-spectrum, then
$\mathbf{1}^E$ is a (commutative) symmetric ring $S^1$-spectrum.
\end{enumerate}
\end{thm}

\begin{proof}
(1). We follow~\cite[\S I.1]{Sch} to verify the relevant conditions for symmetric $S^1$-spectra.
The left action of the symmetric group $\Sigma_n$ on $\Hom_{\cc
C}(P^{\wedge n},E_n\wedge S^n)$ for each $n\geq 0$ is given by
conjugation. In detail, for each $f: P^{\wedge n}\to E_n\wedge S^n$
and each $\tau\in\Sigma_n$ the morphism $\tau\cdot f$ is defined as the composition
   \begin{equation}\label{sigman}
    P^{\wedge n}\xrightarrow{\tau^{-1}}P^{\wedge n}\xrightarrow{f}E_n\wedge S^n
       \xrightarrow{\tau\wedge\tau}E_n\wedge S^n.
   \end{equation}
With this definition each $\Hom_{\cc C}(P^{\wedge n},E_n\wedge
S^n)$ becomes a $\Sigma_n$-simplicial set. Here $\Sigma_n$
acts on $P^{\wedge n}$ and $S^n$ by permutations. 

There are natural morphisms
   $$\bigvee\Hom_{\cc C}(P^{\wedge n},E_n\wedge S^n)\to\Hom_{\cc C}(P^{\wedge (n+k)},
     \bigvee(E_{n+k}\wedge S^{n})),$$
where coproducts are indexed by non-basepoint elements of $S^k_\ell =
\ell_+^{\wedge k}=(\{1,\ldots,\ell\}^{\times k})_+$. They take an element
$f:P^{\wedge n}\to E_n\wedge S^n$ of the $(j_1,\ldots,j_k)$th
summand to the composition
   $$P^{\wedge n+k}\xrightarrow{f\wedge\sigma^{\wedge k}} (E_n\wedge S^n)\wedge T^k\bl{v^{-1}}\cong (E_n\wedge T^k)\wedge S^n
       \to E_{n+k}\wedge S^n\xrightarrow{\iota_{(j_1,\ldots,j_k)}}\bigvee E_{n+k}\wedge S^n.$$
Here $\iota_{(j_1,\ldots,j_k)}$ is the inclusion into the
$(j_1,\ldots,j_k)$th summand. The middle arrow comes from
$E_n\wedge T^k\to E_{n+k}$. It is induced by structure maps of the
symmetric $T$-spectrum $E$ and is
$(\Sigma_n\times\Sigma_k)$-equivariant. Each natural morphism above coincides termwise with
the composite map
   \begin{equation}\label{tutu}
     u_{n+k-1}\circ\cdots\circ
     (u_n\wedge\id):\Hom_{\cc C}(P^{\wedge n},E_n\wedge S^n)\wedge S^k
     \to\Hom_{\cc C}(P^{\wedge (n+k)},E_{n+k}\wedge S^{n+k}).
   \end{equation}

The fact that~\eqref{tutu} is $(\Sigma_n\times\Sigma_k)$-equivariant
follows from commutativity of the diagram
   $$\xymatrix{ P^{\wedge n+k}\ar[d]_{(\tau,\mu)}\ar[r]^(.4){f\wedge\sigma^{\wedge k}}
               &(E_n\wedge S^n)\wedge T^k\ar[r]^{\cong}\ar[d]_{\tau\wedge\tau\wedge\mu}
               &(E_n\wedge T^k)\wedge S^n\ar[r]\ar[d]_{\tau\wedge\mu\wedge\tau}
               &E_{n+k}\wedge S^n\ar[r]^{\iota_{(j_1,\ldots,j_k)}}\ar[d]_{(\tau,\mu)\wedge\tau}
               &\bigvee E_{n+k}\wedge S^n\ar[d]^{(\tau,\mu)\wedge\tau\wedge\mu^{S^k}}\\
               P^{\wedge n+k}
               &(E_n\wedge S^n)\wedge T^k\ar[r]^{\cong}
               &(E_n\wedge T^k)\wedge S^n\ar[r]&E_{n+k}\wedge
               S^n\ar[r]_{\iota_{(j_{\mu(1)},\ldots,j_{\mu(k)})}}
               &\bigvee E_{n+k}\wedge S^n}$$
in which $(\tau,\mu)\in\Sigma_n\times\Sigma_k$, $\mu^{S^k}$
permutes summands of the right upper corner by the rule
$(j_1,\ldots,j_k)\mapsto(j_{\mu(1)},\ldots,j_{\mu(k)})$.
We see that $\mathbf{1}^E$ is a symmetric $S^1$-spectrum.

(2). Suppose $E$ is a symmetric ring $T$-spectrum. Define multiplication maps
   $$\nu_{n,m}:\Hom_{\cc C}(P^{\wedge n},E_n\wedge S^n)\wedge\Hom_{\cc C}(P^{\wedge m},E_m\wedge S^m)
       \to\Hom_{\cc C}(P^{\wedge (n+m)},E_{n+m}\wedge S^{n+m}),\quad n,m\geq 0,$$
by
   $$(f,g)\mapsto(P^{\wedge(n+m)}\xrightarrow{f\wedge g}E_n\wedge S^n\wedge E_m\wedge S^m\cong
       E_n\wedge E_m\wedge S^{n+m}\xrightarrow{\mu_{n,m}\wedge\id}E_{n+m}\wedge S^{n+m}),$$
where $\mu_{n,m}$ is the multiplication map of $E$ (it is $\Sigma_n\times\Sigma_m$-equivariant).

The map $\nu_{n,m}$ is $\Sigma_n\times\Sigma_m$-equivariant. Indeed, this follows from commutativity of the diagram
   $$\xymatrix{P^{\wedge n+m}\ar[d]_{(\tau,\pi)}\ar[r]^(.35){f\wedge g}
               &E_n\wedge S^n\wedge E_m\wedge S^m\ar[r]^{\cong}\ar[d]_{\tau\wedge\tau\wedge\pi\wedge\pi}
               &E_{n}\wedge E_m\wedge S^{n+m}\ar[r]^{\mu_{n,m}\wedge\id}\ar[d]_{\tau\wedge\pi\wedge(\tau,\pi)}
               &E_{n+m}\wedge S^{n+m}\ar[d]^{(\tau,\pi)\wedge(\tau,\pi)}\\
               P^{\wedge n+k}
               &E_{n}\wedge S^n\wedge E_m\wedge S^{m}\ar[r]_{\cong}
               &E_{n}\wedge E_m\wedge S^{n+m}\ar[r]_{\mu_{n,m}\wedge\id}&E_{n+m}\wedge S^{n+m}}$$
in which $(\tau,\pi)\in\Sigma_n\times\Sigma_m$. Clearly, the ``associativity square"
   $$\xymatrix{(P^{\wedge n},E_n\wedge S^n)\wedge(P^{\wedge m},E_m\wedge S^m)
                      \wedge(P^{\wedge p},E_p\wedge S^p)\ar[r]^(.54){\id\wedge\nu_{m,p}}\ar[d]_{\nu_{n,m}\wedge\id}
                      &(P^{\wedge n},E_n\wedge S^n)\wedge(P^{\wedge(m+p)},E_{m+p}\wedge S^{m+p})\ar[d]^{\nu_{n,m+p}}\\
                      (P^{\wedge(n+m)},E_{n+m}\wedge S^{n+m})\wedge(P^{\wedge p},E_{p}\wedge S^{p})\ar[r]^(.55){\nu_{n+m,p}}
                      &(P^{\wedge(n+m+p)},E_{n+m+p}\wedge S^{n+m+p})}$$
is commutative for all $n, m, p\geq 0$ due to associativity of the multiplication maps $\mu_{*,*}$ for $E$.

Next, two unit maps
   \begin{equation}\label{units}
    i_0:S^0\to\Hom_{\cc C}(S,E_0)\quad\textrm{and}\quad i_1:S^1\to\Hom_{\cc C}(P,E_1\wedge S^1)
   \end{equation}
are defined as follows. Let $\iota_0:S\to E_0$ and $\iota_1:T\to E_1$ be
the unit maps for the symmetric ring $T$-spectrum $E$. Then
$i_0$ takes the unbased point of $S^0$ to $\iota_0$ and $i_1$
coincides termwise with the natural morphisms taking an unbased
simplex $j$ of $S^1_\ell=\{0,1,\ldots,\ell\}$ to the composite map
   $$P\xrightarrow{\sigma}T\xrightarrow{\iota_1}E_1\hookrightarrow\bigvee E_1,$$
where the right arrow is the inclusion into the $j$th summand.

The two ``unit composites"
   $$(P^{\wedge n},E_{n}\wedge S^{n})\cong(P^{\wedge n},E_{n}\wedge S^{n})\wedge S^0\xrightarrow{\id\wedge i_0}
       (P^{\wedge n},E_{n}\wedge S^{n})\wedge(S,E_0)\xrightarrow{\nu_{n,0}}(P^{\wedge n},E_{n}\wedge S^{n})$$
   $$(P^{\wedge n},E_{n}\wedge S^{n})\cong S^0\wedge(P^{\wedge n},E_{n}\wedge S^{n})\xrightarrow{i_0\wedge\id}
       (S,E_0)\wedge(P^{\wedge n},E_{n}\wedge S^{n})\xrightarrow{\nu_{0,n}}(P^{\wedge n},E_{n}\wedge S^{n})$$
are the identity for all $n\geq 0$ due to the same properties for $E$.

Furthermore, the ``centrality diagram"
   $$\xymatrix{(P^{\wedge n},E_{n}\wedge S^{n})\wedge S^1\ar[r]^(.4){\id\wedge i_1}\ar[d]_{twist}
                      &(P^{\wedge n},E_{n}\wedge S^{n})\wedge(P^{\wedge 1},E_1\wedge S^1)\ar[r]^(.57){\nu_{n,1}}
                      &(P^{\wedge n+1},E_{n+1}\wedge S^{n+1})\ar[d]^{\chi_{n,1}}\\
                      S^1\wedge(P^{\wedge n},E_{n}\wedge S^{n})\ar[r]^(.4){i_1\wedge\id}
                      &(P^{\wedge 1},E_1\wedge S^{1})\wedge(P^{\wedge n},E_{n}\wedge S^{n})\ar[r]^(.57){\nu_{1,n}}
                      &(P^{\wedge 1+n},E_{1+n}\wedge S^{1+n})}$$
commutes for all $n\geq 0$ due to the same properties for $E$. Here
$\chi_{n,m}\in\Sigma_{n+m}$ denotes the shuffle permutation which
moves the first $n$ elements past the last $m$ elements, keeping
each of the two blocks in order.

It follows that $\mathbf{1}^E$ is a symmetric ring $S^1$-spectrum.
Suppose $E$ is commutative. Then the square
 $$\xymatrix{(P^{\wedge n},E_{n}\wedge S^{n})\wedge(P^{\wedge m},E_m\wedge S^m)\ar[d]_{\nu_{n,m}}\ar[r]^{twist}
                     &(P^{\wedge m},E_m\wedge S^m)\wedge(P^{\wedge n},E_n\wedge S^n)\ar[d]^{\nu_{m,n}}\\
                     (P^{\wedge n+m},E_{n+m}\wedge S^{n+m})\ar[r]^{\chi_{n,m}}&(P^{\wedge m+m},E_{m+n}\wedge S^{m+n})}$$
commutes for all $n,m\geq 0$. We also use here commutativity of the diagram
   $$\xymatrix{P^{\wedge n+m}\ar[r]^(.35){f\wedge g}
                     &E_n\wedge S^n\wedge E_m\wedge S^m\ar[r]^\cong
                     &E_{n}\wedge E_m\wedge S^n\wedge S^m
                     \ar[r]^(.55){\mu_{n,m}\wedge\id}
                     &E_{n+m}\wedge S^{n+m}\\
                     P^{\wedge m+n}\ar[r]^(.35){g\wedge f}\ar[u]_{\chi_{m,n}}
                     &E_{m}\wedge S^m\wedge E_n\wedge S^n\ar[r]^\cong\ar[u]_{twist}
                     &E_{m}\wedge E_n\wedge S^m\wedge S^n\ar[r]^(.55){\mu_{m,n}\wedge\id}\ar[u]^{twist\wedge twist}
                     &E_{m+n}\wedge S^{m+n}\ar[u]_{\chi_{m,n}\wedge\chi_{m,n}}
                     } 
                     $$
for all $f:P^{\wedge n}\to E_n\wedge S^n$ and $g:P^{\wedge m}\to E_m\wedge S^m$. It follows that $\mathbf{1}^E$ is
commutative.
\end{proof}

\begin{cor}\label{symmspcor}
Suppose $E$ is a commutative symmetric ring $T$-spectrum.
The category of right $\mathbf 1^E$-modules $\Mod\mathbf 1^E$ is closed symmetric monoidal,
where $\mathbf 1^E$ is the commutative symmetric ring spectrum of Theorem~\ref{symmsp}.
\end{cor}

\begin{lem}\label{ringmap}
Under the assumptions of Theorem~\ref{symmsp} two unit maps~\eqref{units} can be extended to
a ring morphism between symmetric ring spectra $\delta:\cc S\to\mathbf 1^E$, where $\cc S$ is the sphere spectrum.
\end{lem}

\begin{proof}
This is straightforward.
\end{proof}

\section{Reconstructing stable homotopy theory $SH$}\label{galois}

We refer the reader to~\cite{Nee96} for basic facts on compactly generated triangulated categories. Below we will often use the following lemma.

\begin{lem}[see \cite{GJ}]\label{equ}
Let $\cc S$ and $\cc T$ be compactly generated triangulated categories. Suppose
there exists a set of compact generators $\Sigma$ in $\cc S$ and a triangulated functor 
$F:\cc S\to \cc T$ that preserves direct sums such that
\begin{itemize}
  \item[1.] the collection $\bigl\{F(X)|X\in \Sigma\bigr\}$ is a set of compact generators in $\cc T$\!,
  \item[2.] for any $X,Y$ in $\Sigma$, the induced map 
  $$F_{X,Y[n]}:\Hom_{\cc S} \bigl(X,Y[n]\bigr)\to \Hom_{\cc T}\bigl(FX,FY[n]\bigr)$$ is an isomorphism for all $n\in\bb Z$.
\end{itemize}
Then $F$ is an equivalence of triangulated categories.
\end{lem}

Let $Sp_{S^1,\bb G}(k)$ denote the category of symmetric $(S^1,\bb
G)$-bispectra associated with the closed symmetric monoidal 
category of pointed motivic spaces $\cc M$ (see our notation on p.~\pageref{ntn}).
Here the $\bb G$-direction corresponds to
the pointed motivic space $\bb G$, which is the mapping cone in $\cc M$ of the 
map $1:\pt_+\to(\bb G_m)_+$. 
The category $Sp_{S^1,\bb G}(k)$ is equipped with a
stable motivic model category structure~\cite{Jar}. Denote by $\mathsf{SH}(k)$
its homotopy category. 
If $\cc E$ is a bispectrum and $p,q$ are integers, recall that $\pi_{p,q}(\cc E)$ is the Nisnevich
sheaf of bigraded stable homotopy groups associated to the presheaf
   $$U\in\textrm{Sm}/k\longmapsto \mathsf{SH}(k)(\Sigma^\infty_{S^1}\Sigma^\infty_{\gmp_m^{\wedge 1}}U_+,S^{p-q}\wedge\bb G^{\wedge q}\wedge\cc E).$$
A map of bispectra $f:\cc E\to\cc E'$ is a stable motivic equivalence if and only if $\pi_{*,*}(f)$ is an isomorphism. 
The category $\mathsf{SH}(k)$ has a closed symmetric
monoidal structure with monoidal unit being the motivic sphere
bispectrum $\cc S_k:=\Sigma^\infty_{S^1}\Sigma^\infty_{\bb G} S^0$ (see~\cite{Jar} for details). 

The following theorem was proven by Levine~\cite{Lev} for
algebraically closed fields of characteristic zero, with embedding
$\bar k\hookrightarrow\bb C$ and extended by
Wilson--{\O}stv{\ae}r~\cite[Corollaries 1.2 and 6.5]{WO} to
arbitrary algebraically closed fields.

\begin{thm}\label{LWO}
Let $\bar k$ be an algebraically closed field of exponential
characteristic $e$ and $\cc S$ be the
sphere spectrum $\Sigma^\infty_{S^1}S^0$.
For all $n\geq 0$ the homomorphism $\bb L c:\pi_n(\cc
S)[e^{-1}]\to\pi_{n,0}(\cc S_{\bar k})[e^{-1}]$ is an isomorphism,
where $c:SH\to \mathsf{SH}(\bar k)$ is the functor induced by the functor
$c:\bf S_\bullet\to\cc M$ sending pointed simplicial sets to
constant motivic spaces.
\end{thm}

The following statement was proven by
Zargar~\cite[Theorem~1.1]{Zargar} by using the stable \'etale
realisation functor.

\begin{cor}\label{LWOcor}
Let $\bar k$ be an algebraically closed field of exponential
characteristic $e$. The triangulated functor
   $$\bb L c:SH[1/e]\to \mathsf{SH}(\bar k)[1/e]$$
is full and faithful.
\end{cor}

\begin{proof}
Using Lemma~\ref{equ}, our statement follows from
Theorem~\ref{LWO} if we note that $SH[1/e]$ (respectively the image
of $SH[1/e]$) is compactly generated by $\cc S[1/e]$ (respectively
by $\cc S_{\bar k}[1/e]$).
\end{proof}

Recall from~\cite{GP1} that one of the equivalent ways to define Voevodsky's framed 
correspondences of level $n\geq0$ between smooth $k$-schemes
$X,Y\in\textrm{Sm}/k$ is as follows:
   $$\Fr_n(X,Y):=\Hom_{\cc M}(X_+\wedge\bb P^{\wedge n},Y_+\wedge T^n),$$
where $\bb P^{\wedge n}$ (respectively $T^n$) is the smash product of $n$ copies of $(\bb P^1,\infty)\in\cc M$ 
(respectively $T=\bb A^1/(\bb A^1-\{0\})\in\cc M$). 
   
\begin{rem}
Due to Voevodsky's lemma~\cite[Section~3]{GP1} there is an equivalent description of the sets
$\Fr_n(X,Y)$ in terms of explicit geometric data. In detail, each element in that description is the equivalence class
of a quadruple $(Z,U,\phi,g)$, where $Z$ is a closed subset of $\bb A^n_X$ which is finite over $X$, 
$U$ is an etale neighborhood of $Z$ in $\bb A^n_X$, $\phi=(\phi_1,\ldots,\phi_n)$ is a collection of regular functions
on $U$ such that $\cap_{i=1}^n\{\phi_i=0\}=Z$, and $g$ is a morphism from $U$ to $Y$. The equivalence relation
on such quadruples depends on the choice of the neighborhood $U$.
If the base field is $\bb C$
and $X=Y=\pt$, the Hom-sets $\Hom_{\cc M}(\bb P^{\wedge m},T^n)$, $m,n\geq 0$, can also be described in terms of
{\it holomorphic framed correspondences} -- see~\cite{GG22}. These are
equivalence classes of triples $(Z,U,f)$, where $Z$ consists of finitely many points in $\bb C^m$, $U$ is an open neighborhood
of $Z$, $f=(f_1,\ldots,f_n):U\to\bb C^n$ is a holomorphic map such that $Z=f^{-1}(0)$. The latter description essentially
follows from the Implicit Function Theorem in Complex Analysis.
\end{rem}

Following notation of~\cite{GP1} one sets for any finite pointed set $K$,
   $$\Fr_n(X,Y\otimes K):=\Hom_{\cc M}(X_+\wedge\bb P^{\wedge n},Y_+\wedge T^n\wedge K),\quad X,Y\in\textrm{Sm}/k.$$
There is a distinguished framed correspondence $\sigma:(\bb P^1,\infty)\to T$ in $\Fr_1(\pt,\pt)$ associated with the triple $(\{0\},\bb A^1,t)$.
The external smash product by $\sigma$ gives rise to a map $\Fr_n(X,Y\otimes K)\to \Fr_{n+1}(X,Y\otimes K)$.
One sets,
   $$\Fr(X,Y\otimes K):=\colim(\Fr_0(X,Y\otimes K)\xrightarrow{-\wedge\sigma}\Fr_{1}(X,Y\otimes K))\xrightarrow{-\wedge\sigma}\cdots).$$
Let $\Delta^\bullet_k$ denote the standard cosimplicial affine scheme $n\mapsto\spec(k[x_0,\ldots,x_n]/(x_0+\cdots+x_n-1))$ and
$C_*\Fr(X,Y\otimes K):=\Fr(X\times\Delta^\bullet_k,Y\otimes K)$. By the Additivity Theorem of~\cite{GP1} the assignment
   \begin{equation}\label{gammasp}
    K\in\Gamma^{\op}\mapsto C_*\Fr(X,Y\otimes K)\in\bf S_\bullet
   \end{equation}
gives rise to a special $\Gamma$-space of pointed simplicial sets. Its Segal $S^1$-spectrum $C_*\Fr(X,Y\otimes(\Sigma^\infty_{S^1}S^0))$ is denoted by
$M_{fr}(Y)(X)$ and is called the {\it framed motive of $Y$ evaluated at $X$}.

The following theorem was proven by Garkusha and Panin~\cite{GP1}
for algebraically closed fields of characteristic zero, with
embedding $\bar k\hookrightarrow\bb C$.

\begin{thm}[Garkusha--Panin~\cite{GP1}]\label{aaa}
Let $\bar k$ be an algebraically closed field of characteristic $0$.
Then the framed motive $M_{fr}(\pt)(\pt)$ of the point
$\pt=\spec(\bar k)$ evaluated at $\pt$ has the stable homotopy type
of the classical sphere spectrum $\cc S=\Sigma_{S^1}^\infty S^0$. If
$\bar k$ is an algebraically closed field of positive characteristic
$e>0$ then $M_{fr}(\pt)(\pt)[1/e]$, $\pt=\spec(\bar k)$, has the
stable homotopy type of $\cc S[1/e]$.
\end{thm}

\begin{proof}
The proof literally repeats that of~\cite[Theorem~11.9]{GP1} if we
use Corollary~\ref{LWOcor}.
\end{proof}

\begin{dfn}\label{galoissphere}
Given a field $k$, denote by
   $$\mathbf{1}_k:=(\Fr_0(\pt,\pt),\Fr_1(\pt,S^1),\Fr_2(\pt,S^2),\ldots),\quad\pt=\spec(k),$$
the right $S^1$-spectrum of pointed simplicial sets with structure
maps defined as
   $$\Fr_n(\pt,S^n)\xrightarrow{-\wedge\sigma}\Fr_{n+1}(\pt,S^n)\xrightarrow{}\uhom_{\mathbf S_\bullet}(S^1,\Fr_{n+1}(\pt,S^{n+1})).$$
The {\it sphere spectrum over a field $k$\/} is the $S^1$-spectrum
of pointed simplicial sets
   $$\mathsf{ S}_k:=(\Fr_0(\Delta^\bullet_k,\pt),\Fr_1(\Delta^\bullet_k,S^1),\Fr_2(\Delta^\bullet_k,S^2),\ldots).$$
\end{dfn}

The simplicial set $\Fr_m(\pt,S^m)$ can be regarded as the constant
simplicial pointed set $S^m$ together with the ``coefficients set"
$\Fr_m(\pt,\pt)$ attached to it. Hence the $S^1$-spectrum
$\mathbf{1}_k$ can be regarded as a graded ``tensor algebra"
associated to $S^1$ with ``$\Fr_*(\pt,\pt)$-coefficients". In other
words, we add $\Fr_*(\pt,\pt)$-coefficients to the classical sphere
spectrum $\cc S=(S^0,S^1,\ldots)$.

\begin{prop}\label{rings}
The spectra $\mathbf{1}_k$, $\mathsf{ S}_k$ are commutative symmetric ring
spectra in $Sp^\Sigma_{S^1}$.
\end{prop}

\begin{proof}
We apply Theorem~\ref{symmsp}: $\cc C$ is replaced by the category
of pointed motivic spaces $\cc M$, $P$ (respectively $T$) is
replaced by $\bb P^{\wedge 1}:=(\bb P^1,\infty)$ (respectively by
$T=\bb A^1/(\bb A^1-\{0\})$), $E$ is replaced by the commutative
symmetric sphere $T$-spectrum $\cc S_k=(S^0,T,T^2,\ldots)$. With
this notation $\mathbf{1}_k=\mathbf{1}^E$ of Theorem~\ref{symmsp},
and hence a commutative symmetric ring spectrum.

The commutative symmetric ring structure on $\mathsf{ S}_k$ is
defined in a similar fashion if we use the cosimplicial diagonal
morphism
$\diag:\Delta^\bullet_k\to\Delta_k^\bullet\times\Delta_k^\bullet$.
\end{proof}

Denote by $\Mod{\mathsf{ S}_{k}}$ the category of right $\mathsf{
S}_{k}$-modules in $Sp^\Sigma_{S^1}$. The natural map of ring
objects $\cc S\to{\mathsf{ S}_{k}}$ induces a pair of adjoint
functors
   \begin{equation}\label{adjoint}
     L:Sp^\Sigma_{S^1}\rightleftarrows\Mod{\mathsf{ S}_{k}}:U,
   \end{equation}
where $U$ is the forgetful functor and its left adjoint $L$ is the
``extension of framed scalars" functor.
Following~\cite[Section~4]{SS}, we define the stable model structure
on $\Mod{\mathsf{ S}_{k}}$ by calling a map $f$ of ${\mathsf{
S}_{k}}$-spectra a stable equivalence or fibration if so is $U(f)$.
By~\cite[Theorem~4.1]{SS} this model structure is also cofibrantly
generated monoidal satisfying the monoid axiom. By construction,
$(L,U)$ is a Quillen pair.

\begin{dfn}
The category $\Mod{\mathsf{ S}_{k}}$ is called the {\it category of
framed symmetric $S^1$-spectra over a field $k$}. The {\it stable
homotopy category over a field $k$}, denoted by $SH_k$, is defined
as the homotopy category of $\Mod{\mathsf{ S}_{k}}$ with respect to
the stable model structure. $SH_k$ is a closed symmetric monoidal
category.
\end{dfn}

If $N\in Sp^{\Sigma}_{S^1}$ is a symmetric right $S^1$-spectrum,
define an $S^1$-spectrum
   $$\Fr_*^\Sigma(N):=(N_0,\Hom_{\cc M}(\bb P^{\wedge 1},T\wedge N_1),
       \Hom_{\cc M}(\bb P^{\wedge 2},T^2\wedge N_2),\ldots).$$
Each structure map
   $$\upsilon_n:\Hom_{\cc M}(\bb P^{\wedge n},T^n\wedge N_n)\wedge S^1\to
     \Hom_{\cc M}(\bb P^{\wedge n+1},T^{n+1}\wedge N_{n+1})$$
coincides termwise with the natural morphisms
   $$\bigvee\Hom_{\cc M}(\bb P^{\wedge n},T^n\wedge N_n)\to
     \Hom_{\cc M}(\bb P^{\wedge n+1},T^{n+1}\wedge N_{n+1}),$$
where coproducts are indexed by non-basepoint elements of $S_\ell^1
=\ell_+=\{0,1,\ldots,\ell\}$. They take an element $f:\bb P^{\wedge
n}\to T^n\wedge (N_n)_\ell$ of the $k$th summand to the composition
   \begin{multline*}
       \bb P^{\wedge n+1}\xrightarrow{f\wedge\sigma} (T^n\wedge(N_n)_\ell)\wedge T\cong T^{n+1}\wedge(N_n)_\ell
       \bl{\iota_k}\hookrightarrow\bigvee T^{n+1}\wedge(N_n)_\ell=\\
       =(T^{n+1}\wedge(N_n)_\ell)\wedge S^1_\ell
       \cong T^{n+1}\wedge(N_n\wedge S^1)_\ell\xrightarrow{\id\wedge u_n}
       T^{n+1}\wedge(N_{n+1})_\ell,
   \end{multline*}
where $u_n$ is the $n$th structure map of $N$. We can equivalently
define $\upsilon_n$ by the composition
   \begin{multline*}
     \Hom_{\cc M}(\bb P^{\wedge n},T^n\wedge N_n)\wedge S^1\to\Hom_{\cc M}(\bb P^{\wedge n},T^n\wedge N_n\wedge S^1)
     \xrightarrow{-\wedge\sigma}\Hom_{\cc M}(\bb P^{\wedge n+1},T^n\wedge N_n\wedge S^1\wedge
     T)\cong\\ \cong
     \Hom_{\cc M}(\bb P^{\wedge n+1},T^{n+1}\wedge N_n\wedge S^1)
     \xrightarrow{(u_n)_*}\Hom_{\cc M}(\bb P^{\wedge n},T^n\wedge N_{n+1}).
   \end{multline*}

We also define an $S^1$-spectrum
   $$C_*\Fr_*^\Sigma(N):=(N_0,\Hom_{\cc M}((\Delta_k^\bullet)_+\wedge\bb P^{\wedge 1},T\wedge N_1),
       \Hom_{\cc M}((\Delta_k^\bullet)_+\wedge\bb P^{\wedge 2},T^2\wedge N_2),\ldots)$$
with the structure maps defined as above.

$\Fr_*^\Sigma(N)$ and $C_*\Fr_*^\Sigma(N)$ are symmetric $S^1$-spectra for the
same reason as $\mathbf 1^E$ and $\mathsf{ S}_{k}$ are.

Following notation of~\cite{GNP,GP1} one sets for any finite pointed set $K$ and any integer $k\geq 0$,
   $$\Fr_n(X,(Y\times T^k)\otimes K):=\Hom_{\cc M}(X_+\wedge\bb P^{\wedge n},Y_+\wedge T^{k+n}\wedge K),\quad X,Y\in\textrm{Sm}/k.$$
The external smash product by $\sigma$ gives rise to a map $\Fr_n(X,(Y\times T^k)\otimes K)\to \Fr_{n+1}(X,(Y\times T^k)\otimes K)$.
One sets,
   $$\Fr(X,(Y\times T^k)\otimes K):=\colim(\Fr_0(X,(Y\times T^k)\otimes K)\xrightarrow{-\wedge\sigma}\Fr_{1}(X,(Y\times T^k)\otimes K))\xrightarrow{-\wedge\sigma}\cdots)$$
and 
   $$C_*\Fr(X,(Y\times T^k)\otimes K):=\Fr(X\times\Delta^\bullet_k,(Y\times T^k)\otimes K).$$ 
By the Additivity Theorem of~\cite{GP1} the assignment
   \begin{equation*}\label{gammasptk}
    K\in\Gamma^{\op}\mapsto C_*\Fr(X,(Y\times T^k)\otimes K)\in\bf S_\bullet
   \end{equation*}
gives rise to a special $\Gamma$-space of pointed simplicial sets. The pointed motivic space
$X\in\textrm{Sm}/k\mapsto C_*\Fr(X,(Y\times T^k)\otimes K)\in\bf S_\bullet$ is denoted by $C_*\Fr((Y\times T^k)\otimes K)$ in~\cite{GNP,GP1}.
If $Y=\pt$ the latter motivic space is denoted by $C_*\Fr(T^k\otimes K)$. 

\begin{dfn}\label{sigmafrmot}
The {\it symmetric framed motive\/} of a symmetric $S^1$-spectrum
$N\in Sp^\Sigma_{S^1}$ is the symmetric $S^1$-spectrum
   $${M}^\Sigma_{fr}(N):=(C_*\Fr(\pt,\pt\otimes N_0),\Hom_{\cc M}(\bb P^{\wedge 1},C_*\Fr(T\otimes N_1)),
       \Hom_{\cc M}(\bb P^{\wedge 2},C_*\Fr(T^2\otimes N_2)),\ldots)$$
with structure maps and actions of the symmetric groups defined similarly to $C_*\Fr_*^\Sigma(N)$. The framed
motive of the suspension spectrum $\Sigma_{S^1}^\infty X$ of a
pointed simplicial set $X$ will be denoted by ${M}^\Sigma_{fr}(X)$.
\end{dfn}

\begin{rem}\label{mfr}
If $N$ is the suspension symmetric spectrum $\Sigma_{S^1}^\infty X$ of a
pointed simplicial set $X$, the framed motive in the sense of the
preceding definition is a bit different from the framed motive of
$X$ evaluated at $\pt$ defined as
   $${M}_{fr}(X):=C_*\Fr(\pt,\pt\otimes(\Sigma_{S^1}^\infty X)),$$
where $K\in\Gamma^{\op}\mapsto C_*\Fr(\pt,\pt\otimes K)\in\bf S_\bullet$ is the special $\Gamma$-space~\eqref{gammasp}.
\end{rem}

\begin{lem}\label{mfrsymm}
If $k$ is perfect field, then the canonical map of ordinary
$S^1$-spectra $M_{fr}(X)\to{M}^\Sigma_{fr}(X)$ is a level
equivalence in positive degrees for any pointed simplicial set $X$.
\end{lem}

\begin{proof}
Repeating the proof of~\cite[Lemma~4.12]{GN} word for word, we have
that the canonical map of connected spaces
   $$C_*\Fr(\pt,\pt\otimes(X\wedge S^\ell))\to\Hom_{\cc M}(\bb P^{\wedge k},C_*\Fr(T^k\otimes( X\wedge S^\ell))),\quad k,\ell>0,$$
is a weak equivalence. It follows that the map
$M_{fr}(X)\to{M}^\Sigma_{fr}(X)$ is a level equivalence of spectra
in positive degrees.
\end{proof}

The reader should not confuse $\pi_*$-isomorphisms (i.e. maps
inducing isomorphisms of stable homotopy groups) and stable
equivalences of symmetric spectra. The first class is a proper
subclass of the second.

Though ${M}_{fr}(X)$ is canonically a symmetric
$S^1$-spectrum, where $\Sigma_n$ acts on each space by permuting
$S^n$, the point is that it is not a $\mathsf{S}_k$-module in
contrast with ${M}^\Sigma_{fr}(X)$.

\begin{prop}\label{frmodule}
Given a field $k$ and $N\in Sp_{S^1}^\Sigma$, the following
statements are true:
\begin{enumerate}
\item $C_*\Fr_*^\Sigma(N)$ and ${M}^\Sigma_{fr}(N)$ are right $\mathsf{S}_{k}$-modules;
\item every map of symmetric $S^1$-spectra $f:N\to N'$
induces morphisms of right $\mathsf{ S}_{k}$-modules
$C_*\Fr_*^\Sigma(f):C_*\Fr_*^\Sigma(N)\to C_*\Fr_*^\Sigma(N')$ and
${M}_{fr}^\Sigma(f):{M}_{fr}^\Sigma(N)\to {M}_{fr}^\Sigma(N')$;
\item the canonical map
$\alpha:{\mathsf{S}_{k}}\to M_{fr}^\Sigma(S^0)$ in $\Mod
\mathsf{S}_k$ is a $\pi_*$-isomorphism (i.e. a stable equivalence of
ordinary spectra) whenever the base field $k$ is perfect.
\item for every $N\in Sp_{S^1}^\Sigma$ the canonical map
$\beta:C_*\Fr_*^\Sigma(N)\to{M}^\Sigma_{fr}(N)$
is a $\pi_*$-isomorphism whenever the base field $k$ is perfect.
\end{enumerate}
\end{prop}

\begin{proof}
(1). The desired pairing
   $$(C_*\Fr_*^\Sigma(N)\wedge\mathsf{S}_k)_m=\bigvee_{q+p=m}(\Sigma_{q+p})_+
     \wedge_{\Sigma_q\times\Sigma_p}C_*\Fr_q(N_q)\wedge\Fr_p(\Delta^\bullet,S^p)\to C_*\Fr_m(N_m)$$
is defined as follows. Given two morphisms
$(\beta:\Delta^\ell_+\wedge\bb P^{\wedge q} \to T^q\wedge
N_q)\in\Fr_q(\Delta^\ell,N_q)$ and $(\alpha:\Delta^\ell_+\wedge\bb
P^{\wedge p} \to T^p\wedge S^p)\in\Fr_p(\Delta^\ell,S^p)$, define a
morphism $\beta\star\alpha\in\Fr(\Delta^\ell,N_{q+p})$ as the
composite
   \begin{multline*}
    (\Delta^\ell)_+\wedge\bb P^{\wedge (q+p)}\xrightarrow{diag\wedge\id}
    (\Delta^\ell\times\Delta^\ell)_+\wedge\bb P^{\wedge (q+p)}\cong
    (\Delta^\ell_+\wedge\bb P^{\wedge q})\wedge(\Delta^\ell_+\wedge\bb P^{\wedge p})\xrightarrow{\beta\wedge\alpha} \\
    (T^q\wedge N_q)\wedge(T^p\wedge S^p)
    \cong T^{q+p}\wedge N_q\wedge S^p\to T^{q+p}\wedge N_{q+p}.
   \end{multline*}
It is straightforward to see that this pairing is
$\Sigma_q\times\Sigma_p$-equivariant, satisfies associativity and
unit conditions, hence it defines the structure of a right $\mathsf{
S}_{k}$-module on $C_*\Fr_*^\Sigma(N)$. For the same reasons,
${M}^\Sigma_{fr}(N)$ is a right $\mathsf{S}_{k}$-module.

(2). This is straightforward.

(3). Let $\Theta^\infty_{S^1}$ be the naive stabilisation functor of
$S^1$-spectra. It has the property that $X\to\Theta^\infty_{S^1}(X)$
is a stable equivalence for every $S^1$-spectrum
$X$~\cite[Proposition~4.7]{H}.

Let $A_{n,r}:=C_*(\bb P^{\wedge n},T^n\wedge S^r)$,
$B_{m,n,r}:=C_*(\bb P^{\wedge(m+n)},T^{m+n}\wedge S^r)$, $m,n,r\geq
0$. We have maps of spaces
   $$A_{n,r}\xrightarrow{\sigma}A_{n+1,r}\quad\textrm{and}\quad
     B_{m,n,r}\xrightarrow{\sigma}B_{m+1,n,r}\xrightarrow{\sigma}B_{m+1,n+1,r}.$$
Let $A_r:=\colim_{n\geq r} A_{n,r}$, $B_{m,r}:=\colim_{n}B_{m,n,r}$,
$B_r:=\colim_{m,n}B_{m,n,r}$. The spaces $A_r$, $B_r$ constitute
right $S^1$-spectra $A$ and $B$. 
Each structure map
$A_r\to\uhom(S^1,A_{r+1})$ is the composite map determined by the following commutative diagram
   $$\xymatrix{A_r:&(\bb P^{\wedge r},T^r\wedge S^{r})\ar[d]_\sigma\ar[r]^{\sigma}
                             &(\bb P^{\wedge(r+1)},T^{r+1}\wedge S^{r})\ar[d]_\sigma\ar[r]^(.7){\sigma}&\cdots\\
                             &(\bb P^{\wedge (r+1)},T^{r+1}\wedge S^r)\ar[d]_{-\wedge S^1}\ar[r]^{\sigma}
                             &(\bb P^{\wedge(r+2)},T^{r+2}\wedge S^r)\ar[d]_{-\wedge S^1}\ar[r]^(.7){\sigma}&\cdots\\
   \uhom(S^1,A_{r+1}):&(S^1,(\bb P^{\wedge (r+1)},T^{r+1}\wedge S^{r+1}))\ar[r]^{(\sigma,S^1)}
   &(S^1,(\bb P^{\wedge(r+2)},T^{r+2}\wedge S^{r+1}))\ar[r]^(.8){(\sigma,S^1)}&\cdots}$$
(we omit $C_*$ here and below for brevity). Each structure map $B_r\to\uhom(S^1,B_{r+1})$ is defined 
in a similar fashion.

There is a commutative diagram of
spectra
   $$\xymatrix{{\mathsf{S}_{k}}\ar@/^1pc/[rr]^{j_{\mathsf{S}_{k}}}\ar[d]_\alpha\ar[r]_a
                &A\ar[d]_c\ar[r]_{j_A}&\Theta^\infty_{S^1}(A)\ar[d]^{\Theta^\infty_{S^1}(c)}\\
                M_{fr}^\Sigma(S^0)\ar@/_1pc/[rr]_{j_{M_{fr}^\Sigma(S^0)}}\ar[r]^b&B\ar[r]^{j_B}&\Theta^\infty_{S^1}(B)}$$
in which
$\Theta^\infty_{S^1}(A)=\Theta^\infty_{S^1}({\mathsf{S}_{k}})$,
$\Theta^\infty_{S^1}(B)=\Theta^\infty_{S^1}(M_{fr}^\Sigma(S^0))$.
The maps $a,b,c$ are defined in a canonical way. By the
two-out-of-three property $a,b$ are stable equivalences. Therefore
$\alpha$ is a stable equivalence if and only if $c$ is.

But $c$ is the infinite composition
   $$A_r\cong B_{0,r}\to B_{1,r}\to B_{2,r}\to\cdots B_r.$$
Each composition $A_r\to B_{m,r}$ is isomorphic to the canonical map
of spaces
   $$C_*\Fr(S^r)\to\Hom_{\cc M}(\bb P^{\wedge m},C_*\Fr(T^m\otimes S^r)).$$
Repeating the proof of~\cite[Lemma~4.12]{GN} word for word, the map
is a weak equivalence for positive $r$, and hence $c$ is a weak
equivalence in positive degrees.

(4). The proof literally repeats that of (3) if we replace $S^r$ by
$N_r$.
\end{proof}

Recall that we distinguish classical (also called naive) stable
homotopy groups $\pi_*(N)$ of $N\in Sp^\Sigma_{S^1}$ and ``true" stable
homotopy groups (denote them by $\underline{\pi}_*(N)$ --- see,
e.g., \cite{Sch} for details). They coincide for semistable
symmetric spectra.

\begin{cor}
Given a perfect field $k$, the symmetric spectrum $\mathsf{S}_{k}$
is semistable and
$\pi_0(\mathsf{S}_{k})=\underline{\pi}_0(\mathsf{S}_{k})=K_0^{MW}(k)$,
where $K_0^{MW}(k)$ is the Milnor-Witt group of $k$.
\end{cor}

\begin{proof}
By~\cite[Corollary~11.2]{GP1} $\pi_{0,0}(\Sigma^\infty_{S^1}\Sigma^\infty_{\bb
G_m}S^0)=\pi_0(M_{fr}(\pt))$. By a theorem of Morel~\cite{Mor1} one
has $\pi_{0,0}(\Sigma^\infty_{S^1}\Sigma^\infty_{\bb
G_m}S^0)=K_0^{MW}(k)$. By definition~\cite{GP1}, $M_{fr}(\pt)$ is
the Segal symmetric spectrum associated with a special
$\Gamma$-space. Therefore $M_{fr}(\pt)$ is semistable
by~\cite[Example~4.2]{Sch08}. Our statement now follows from
Lemma~\ref{mfrsymm} and Proposition~\ref{frmodule}(3).
\end{proof}

By $e^{-1}$-stable equivalences (respectively
$\pi_*[e^{-1}]$-isomorphisms) we mean maps of symmetric spectra
(respectively ordinary spectra) inducing isomorphisms in
$SH[e^{-1}]$ (respectively isomorphisms of stable homotopy groups
with $e^{-1}$-coefficients).

\begin{thm}\label{compos}
If $\bar k$ is an algebraically closed field of exponential
characteristic $e$, then the natural maps of symmetric $S^1$-spectra
   $$\nu:N\xrightarrow{}C_*\Fr_*^\Sigma(N),\quad \beta\circ\nu:N\xrightarrow{}M^\Sigma_{fr}(N)$$
are $\pi_*[e^{-1}]$-isomorphisms, where
$\beta:C_*\Fr_*^\Sigma(N)\xrightarrow{}M^\Sigma_{fr}(N)$ is a canonical
map. In particular, a map $\gamma:N\to N'$ of symmetric
$S^1$-spectra is an $e^{-1}$-stable equivalence or a
$\pi_*[e^{-1}]$-isomorphism if and only if ${M}^\Sigma_{fr}(\gamma)$
is.
\end{thm}

\begin{proof}
The map $\beta$ is a $\pi_*$-isomorphism by
Proposition~\ref{frmodule}(4). Consider a commutative diagram
   $$\xymatrix{N\ar[dr]_{\kappa_N}\ar[r]^(0.4)\nu&C_*\Fr_*^\Sigma(N)\ar[d]_a\ar[dr]^{j_{C_*\Fr_*(N)}}\\
               &M_{fr}(N)\ar[r]_{j_{M_{fr}(N)}}&\Theta^\infty_{S^1} M_{fr}(N)}$$
We see that $a$ is a $\pi_*$-isomorphism. Therefore $\nu$ is a
$\pi_*[e^{-1}]$-isomorphism if and only if $\kappa_N$ is.

We claim that the map of $S^1$-spectra
$\varkappa_X:\Sigma_{S^1}^\infty X\to {M}_{fr}(X)$ is a
$\pi_*[e^{-1}]$-isomorphism. Indeed, if $X$ is a finite pointed set
regarded as a constant simplicial pointed set, this follows from
Theorem~\ref{aaa}. If $X$ is any pointed set, then $\varkappa_X$ is
a directed colimit of maps $\varkappa_W$, where $W$ runs over finite
pointed sets of $X$. Hence $\varkappa_X$ is a
$\pi_*[e^{-1}]$-isomorphism as directed colimits preserve
$\pi_*[e^{-1}]$-isomorphisms. Finally, since the geometric
realization of a simplicial $\pi_*[e^{-1}]$-isomorphism is an
$\pi_*[e^{-1}]$-isomorphism, then so is $\kappa_X$ for an arbitrary
pointed simplicial set $X$ as claimed.

Next, every (ordinary) spectrum $N\in Sp_{S^1}$ equals
$\colim_iL_i(N)$, where each spectrum
$L_i(N)=(N_0,\ldots,N_{i-1},N_i,N_i\wedge S^1,N_i\wedge
S^2,\ldots\ldots)$. Then $\varkappa_N:N\to {M}_{fr}(N)$ equals
$\kappa_N=\colim_i\kappa_{L_i(N)}$. It is a
$\pi_*[e^{-1}]$-isomorphism as each $\kappa_{L_i(N)}$ is (this
follows from the previous claim about $\kappa_X$). Therefore the
natural map of $S^1$-spectra $\kappa_N:N\to {M}_{fr}(N)$ is a
$\pi_*[e^{-1}]$-isomorphism for all $N\in Sp_{S^1}$.
\end{proof}

We are now in a position to prove the main result of this section
saying that the stable homotopy category of classical symmetric
spectra can be recovered from the stable homotopy category of framed
spectra over an algebraically closed field (after inverting the
exponential characteristic).

\begin{thm}\label{recover}
Suppose $k=\bar k$ is an algebraically closed field of exponential
characteristic $e$. The Quillen pair $(L,U)$~\eqref{adjoint} is a
Quillen equivalence. In particular, it induces an equivalence of
compactly generated triangulated categories
   $$L:SH[e^{-1}]\rightleftarrows SH_k[e^{-1}]:U.$$
Moreover, the equivalence $L$ is isomorphic to the functor
   $$M_{fr}^\Sigma:SH[e^{-1}]\xrightarrow{\sim} SH_k[e^{-1}]$$
that takes a symmetric $S^1$-spectrum $N$ to its symmetric framed
motive $M_{fr}^\Sigma(N)$.
\end{thm}

\begin{proof}
$SH=\Ho(Sp^\Sigma_{S^1})$ is compactly generated by the sphere
spectrum $\cc S$. $SH_k$ is compactly generated by the framed sphere
spectrum ${\mathsf{ S}_{k}}$. By construction, $L(\cc
S)={\mathsf{S}_{k}}$. Therefore our statement that $(L,U)$ is a
Quillen equivalence reduces to showing that the composite map of
$S^1$-spectra is a $e^{-1}$-stable equivalence
   $$\phi:\cc S\to UL(\cc S)=U({\mathsf{S}_{k}})\to U({\mathsf{S}_{k}^f}),$$
where ${\mathsf{S}_{k}^f}$ is a fibrant replacement of
${\mathsf{S}_{k}}$ in $\Mod{\mathsf{ S}_{k}}$ (we also
use Lemma~\ref{equ} here).

We claim that ${\mathsf{ S}^f_{k}}=\Omega_{S^1}M_{fr}^\Sigma(S^1)$.
Indeed, the canonical map $\alpha:{\mathsf{S}_{k}}\to
M_{fr}^\Sigma(S^0)$ in $\Mod{\mathsf{ S}_{k}}$ is a
$\pi_*$-isomorphism by Proposition~\ref{frmodule}(3). It follows
from~\cite[Theorem~4.1]{GP1} that $M_{fr}^\Sigma(S^0)$ is a
positively fibrant symmetric $\Omega$-spectrum, and hence
$\Omega_{S^1}M_{fr}^\Sigma(S^1)$ is an $\Omega$-spectrum. It follows
that ${\mathsf{S}^f_{k}}=\Omega_{S^1}M_{fr}(S^1)$.
Theorem~\ref{compos} implies that $\phi$ is a
$\pi_*[e^{-1}]$-isomorphism.

Next, Theorem~\ref{compos} implies that $N\mapsto M_{fr}^\Sigma(N)$
induces a functor
   $$M_{fr}^\Sigma:SH[e^{-1}]\xrightarrow{} SH_k[e^{-1}].$$
We have that $M_{fr}^\Sigma(\cc S)\cong \mathsf{S}_{k}$ is a compact
generator. The first part of the proof implies that
   $$SH[e^{-1}](\cc S[*],\cc S)\to SH_k[e^{-1}](M^\Sigma_{fr}(\cc S)[*],M^\Sigma_{fr}(\cc S))$$
is an isomorphism of graded Abelian groups, hence $M_{fr}^\Sigma$ is
an equivalence of compactly generated triangulated categories
by Lemma~\ref{equ}. By Theorem~\ref{compos} the canonical map
of $S^1$-spectra $N\to M_{fr}^\Sigma(N)$ is a
$\pi_*[e^{-1}]$-isomorphism for any $N\in Sp_{S^1}^\Sigma$.
Therefore $\id\to U\circ M_{fr}^\Sigma$ is an isomorphism of
functors. Composing it with $(M^\Sigma_{fr})^{-1}$, where
$(M^\Sigma_{fr})^{-1}$ is a quasi-inverse functor to
$M^\Sigma_{fr}$, we get an isomorphism of functors
$(M^\Sigma_{fr})^{-1}\simeq U$. By the first part of the proof $L$
is a quasi-inverse functor to $U$, and hence $L$ is isomorphic to
the functor $M^\Sigma_{fr}$, as was to be shown.
\end{proof}

\section{Spectral categories associated to symmetric ring spectra}\label{spectral}

In what follows by a {\it spectral category\/} we mean a category enriched over the
closed symmetric monoidal category of symmetric $S^1$-spectra $Sp^\Sigma_{S^1}$.

Recall from~\cite[Construction~5.6]{Sch} that for every pair of symmetric spectra $X, Y$
a morphism $X\wedge Y\to Z$ to a symmetric spectrum $Z$ is the same as giving a
bimorphism $b:(X,Y)\to Z$. We define a {\it bimorphism $b : (X, Y )\to Z$}
as a collection of $\Sigma_p\times\Sigma_q$-equivariant maps of pointed simplicial sets
$b_{p,q}:X_p\wedge Y_q\to Z_{p+q}$ for $p,q\geq 0$, such that the ``bilinearity diagram"
commutes for all $p,q\geq 0$:
   \begin{equation}\label{bimorphism}
     \xymatrix{&X_p\wedge Y_q\wedge S^1\ar[d]^{b_{p,q}\wedge S^1}\ar[rr]^{X_p\wedge twist}\ar[dl]_{X_p\wedge\sigma_q}
                      &&X^p\wedge S^1\wedge Y_q\ar[d]^{\sigma_p\wedge Y_q}\\
                      X_p\wedge Y_{q+1}\ar[dr]_{b_{p,q+1}}&Z_{p+q}\wedge S^1\ar[d]^{\sigma_{p+q}}&&X_{p+1}\wedge Y_q\ar[d]^{b_{p+1,q}}\\
                      &Z_{p+q+1}&&Z_{p+1+q}\ar[ll]^{1\times\chi_{1,q}}}
   \end{equation}

In this section $\cc C$ is the category from Section~\ref{graded}.

\begin{dfn}
Suppose $\cc B$ is a full subcategory of $\cc C$ closed under $\wedge$ and $\sigma:P\to T$ is a morphism in $\cc C$.
Let $E$ be a symmetric ring $T$-spectrum in
$Sp^\Sigma(\cc C,T)$. We define the {\it symmetric $S^1$-spectrum of $(E,\sigma)$-correspondences
$\corr^{E,\sigma}_*(X,Y)$\/} between two objects $X,Y\in\cc B$ as follows.
First, let
   $$\corr^{E,\sigma}_n(X,Y):=\Hom_{\cc C}(X\wedge P^{\wedge n},Y\wedge E_n\wedge S^n).$$
This simplicial set is pointed at the zeroth map. 
By definition, $\corr^{E,\sigma}_0(X,Y):=\Hom_{\cc C}(X,Y\wedge E_0)$.
Similarly to~\eqref{sigman} each $\corr^{E,\sigma}_n(X,Y)$ is a $\Sigma_n$-simplicial set. 
The left action of $\Sigma_n$ on
$\corr^{E,\sigma}_n(X,Y)$ is given by conjugation: for each $f:X\wedge P^{\wedge n}\to Y\wedge E_n\wedge S^n$
and each $\tau\in\Sigma_n$ the morphism $\tau\cdot f$ is defined as
the composition
   \begin{equation*}\label{sigmaseqq}
    X\wedge P^{\wedge n}\xrightarrow{X\wedge\tau^{-1}}X\wedge P^{\wedge n}\xrightarrow{f}Y\wedge E_n\wedge S^n
       \xrightarrow{Y\wedge\tau\wedge\tau}Y\wedge E_n\wedge S^n.
   \end{equation*}

Second, repeating the proof of Theorem~\ref{symmsp}(1) word for word 
the morphism $\sigma$ induces natural $(\Sigma_n\times\Sigma_k)$-equivariant maps 
   $$\corr^{E,\sigma}_n(X,Y)\wedge S^k\to\corr^{E,\sigma}_{n+k}(X,Y),$$
so that 
   $$\corr^{E,\sigma}_*(X,Y):=(\corr^{E,\sigma}_0(X,Y),\corr^{E,\sigma}_1(X,Y),\corr^{E,\sigma}_2(X,Y),\ldots)$$
becomes a symmetric $S^1$-spectrum.
\end{dfn}

Define a pairing
   $$\phi_{X,Y,Z}^\sigma:\corr^{E,\sigma}_n(X,Y)\wedge\corr^{E,\sigma}_m(Y,Z)\to\corr^{E,\sigma}_{n+m}(X,Z)$$
by the rule: $\phi_{X,Y,Z}^\sigma(f:X\wedge P^{\wedge n}\to Y\wedge E_n\wedge S^n,g:Y\wedge P^{\wedge m}\to Z\wedge E_m\wedge S^m)$
is given by the composition
\begin{gather*}
X\wedge P^{\wedge n}\wedge P^{\wedge m}\xrightarrow{f\wedge P^{\wedge m}}
Y\wedge E_n\wedge S^n\wedge P^{\wedge m}
\xrightarrow{tw} Y\wedge P^{\wedge m}\wedge E_n\wedge S^n\xrightarrow{g\wedge E_n}\\
Z\wedge E_{m}\wedge S^m\wedge E_n\wedge S^n\xrightarrow{tw} Z\wedge
E_n\wedge E_m\wedge S^n\wedge S^m\xrightarrow{Z\wedge
\mu_{n,m}\wedge S^n\wedge S^m}Z\wedge E_{n+m}\wedge S^{n+m}.
\end{gather*}

\begin{thm}\label{tampa}
Let $E$ be a symmetric ring $T$-spectrum in $Sp^\Sigma(\cc C,T)$ and
$\cc B$ is a full subcategory of $\cc C$ closed under monoidal
product. Then $\cc B$ can be enriched over the closed symmetric
monoidal category of symmetric $S^1$-spectra $Sp_{S^1}^\Sigma$. Namely,
$Sp^\Sigma_{S^1}$-objects of morphisms are defined by the symmetric spectra $\corr^{E,\sigma}_*(X,Y)$ of $(E,\sigma)$-correspondences.
Compositions are defined by pairings $\phi_{X,Y,Z}$. The resulting
$Sp^\Sigma_{S^1}$-category is denoted by $\corr^{E,\sigma}_*(\cc B)$.
Moreover, the $Sp^\Sigma_{S^1}$-category $\corr^{E,\sigma}_*(\cc B)$ is
symmetric monoidal with the same monoidal product on objects as in
$\cc B$ whenever $E$ is a commutative ring $T$-spectrum.
\end{thm}

\begin{proof}
The identity morphism is defined by
   $$X\bl{\rho^{-1}}\cong X\wedge S\xrightarrow{\id_X\wedge\iota_{0}}X\wedge E_0\in\corr_0^E(X,X),$$
where $\iota_0:S\to E_0$ is the unit map. Our proof now literally
repeats that of Theorem~\ref{capitals}. The only thing one has to
care about is that the pairings occurring here are bimorphisms of
symmetric spectra. This is clearly the case if one chases over the
diagram~\eqref{bimorphism}. It is also worth noting that
$(E_0,E_1\wedge S^1,E_2\wedge S^2,\ldots)$ is a ring object in the
category of simplicial symmetric sequences in $\cc C$, which is
commutative whenever $E$ is.
\end{proof}

The proof of the theorem implies the following result.

\begin{cor}
Under the notation of Theorem~\ref{tampa}
any morphism of two symmetric ring $T$-spectra $\gamma:E\to E'$ in $Sp^\Sigma(\cc C,T)$ induces a morphism
of spectral categories
$\gamma_*:\corr^{E,\sigma}_*(\cc B)\to\corr^{E',\sigma}_*(\cc B)$.
\end{cor}

Day's Theorem~\cite{Day} together with Theorem~\ref{tampa} also imply the following result.

\begin{cor}
Under the notation of Theorem~\ref{tampa}
if $E$ is a commutative symmetric ring $T$-spectrum then the category of right
$\corr^{E,\sigma}_*(\cc B)$-modules is a closed symmetric
monoidal category.
\end{cor}

\begin{thm}\label{tampamore}
Suppose $E$ is a commutative symmetric ring $T$-spectrum in $Sp^\Sigma(\cc C,T)$.
Under the assumptions of Theorem~\ref{tampa} the spectral category $\corr^{E,\sigma}_*(\cc B)$ is also
a symmetric monoidal $\Mod\mathbf 1^E$-category with the same monoidal product on objects as in $\cc B$, where
$\Mod\mathbf 1^E$ is the closed symmetric monoidal category of Corollary~\ref{symmspcor}.
\end{thm}

\begin{proof}
Each symmetric spectrum $\corr^{E,\sigma}_*(\cc B)(U,V)$, $U,V\in\Ob\cc B$, is canonically in $\Mod\mathbf 1^E$
if we define compositions $\theta_{p,r}:\corr^{E,\sigma}_p(\cc B)(U,V)\wedge\mathbf 1^E_r\to\corr^{E,\sigma}_{p+r}(\cc B)(U,V)$ by
\begin{gather*}
U\wedge P^{\wedge p}\wedge P^{\wedge r}\xrightarrow{f\wedge g}
V\wedge E_p\wedge S^p\wedge E_r\wedge S^r
\xrightarrow{tw} V\wedge E_n\wedge E_r\wedge S^{n+r}
\xrightarrow{V\wedge \mu_{n,r}\wedge S^{n+r}}V\wedge E_{n+r}\wedge S^{n+r},
\end{gather*}
where $f\in\corr^{E,\sigma}_p(\cc B)(U,V)$, $g\in\mathbf 1^E_r$ and $\mu_{*,*}$ is the multiplication
map of $E$.
The proof is like that of Theorem~\ref{tampa} if we observe that the diagram
   \begin{equation}\label{bimorphismm}
     \xymatrix{&X_p\wedge Y_q\wedge\mathbf 1^E_r\ar[d]^{b_{p,q}\wedge \mathbf 1^E_r}\ar[rr]^{X_p\wedge twist}\ar[dl]_{X_p\wedge\theta_{q,r}}
                      &&X^p\wedge\mathbf 1^E_r\wedge Y_q\ar[d]^{\theta_{p,r}\wedge Y_q}\\
                      X_p\wedge Y_{q+r}\ar[dr]_{b_{p,q+r}}&Z_{p+q}\wedge\mathbf 1^E_r\ar[d]^{\theta_{p+q,r}}&&X_{p+r}\wedge Y_q\ar[d]^{b_{p+r,q}}\\
                      &Z_{p+q+r}&&Z_{p+r+q}\ar[ll]^{1\times\chi_{r,q}}}
   \end{equation}
is commutative with $X=\corr^{E,\sigma}_*(\cc B)(U,V),Y=\corr^{E,\sigma}_*(\cc B)(V,W),Z=\corr^{E,\sigma}_*(\cc B)(U,W)$
and $b=\phi_{U,V,W}$.
\end{proof}

\begin{cor}
Under the notation of Theorem~\ref{tampamore}
any morphism of two symmetric ring $T$-spectra $\gamma:E\to E'$ in $Sp^\Sigma(\cc C,T)$ induces a morphism
of $\Mod\mathbf 1^E$-categories
$\gamma_*:\corr^{E,\sigma}_*(\cc B)\to\corr^{E',\sigma}_*(\cc B)$.
\end{cor}

Day's Theorem~\cite{Day} together with Theorem~\ref{tampa} also imply the following result.

\begin{cor}\label{tampamorecor}
Under the notation of Theorem~\ref{tampamore}
if $E$ is a commutative symmetric ring $T$-spectrum then the category of 
$\corr^{E,\sigma}_*(\cc B)$-modules in the category $\Mod\mathbf 1^E$ is a closed symmetric
monoidal category.
\end{cor}

\section{Enriched motivic homotopy theory}\label{enriched}

One of the approaches to Morel--Voevodsky's stable motivic homotopy theory $SH(k)$ over a field $k$ 
is by means of symmetric $T$-spectra $Sp_T^\Sigma(k)$, where $T=\bb A^1/(\bb A^1-\{0\})$ (see, e.g.,~\cite{Jar}). In detail, we start with motivic spaces
$\cc M$ equipped with the flasque motivic model structure in the sense of~\cite{Is} and then pass to
$Sp_T^\Sigma(k)$ equipped with the stable model structure. The homotopy category of the latter model category
is denoted by $SH(k)$.

A genuinely local approach to $SH(k)$, envisioned by Voevodsky in 2001, 
is presented in~\cite{GP5}. It is based on Voevodsky's
framed correspondences and the machinery of framed motives~\cite{GP1}. 

In this section we suggest yet another (genuinely local) approach to $SH(k)$
and, more generally, a local model for the category of $E$-modules
in $SH(k)$, where $E$ is a symmetric Thom ring spectrum. It is an application of enriched
category theory of spectral categories and spectral modules of Section~\ref{spectral}. The same approach was used 
in~\cite{GP12,GP14} to construct the theory of $K$-motives. 

Following~\cite{GN}, a symmetric $T$-spectrum $E$ is called a {\it Thom spectrum}
if each motivic space $E_n$ has the form
\[E_n=\colim_i E_{n,i}, \text{  }E_{n,i}=V_{n,i}/(V_{n,i}-Z_{n,i}),\]
where $V_{n,i}\to V_{n,i+1}$ is a directed sequence of smooth
varieties, $Z_{n,i}\to Z_{n,i+1}$ is a directed system of smooth
closed subschemes in $V_{n,i}.$ We say that a Thom spectrum $E$ {\it
has the bounding constant $d$\/} if $d$ is the minimal integer such
that codimension of $Z_{n,i}$ in $V_{n,i}$ is strictly greater than
$n-d$ for all $i,n$. The $T$-spectrum $E$ is said to be
a {\it spectrum with contractible alternating group action}, if for
any $n$ and any even permutation $\tau\in \Sigma_n$ there is an
$\bb A^1$-homotopy $E_n\to \uhom(\bb A^1,E_n)$ between the action of
$\tau$ and the identity map. In other words, $E$ neglects the action
of even permutations up to $\bb A^1$-homotopy.

Unless it is specified otherwise $E$ is a symmetric Thom ring $T$-spectrum with the bounding constant $d\leq 1$ 
and contractible alternating group action throughout this section. By~\cite[Lemma~10.2]{GN} $E\in SH^{\eff}(k)$,
where $SH^{\eff}(k)$ is the full triangulated subcategory of $SH(k)$ of effective $T$-spectra.
It is compactly generated by the suspension $T$-spectra
$\Sigma^\infty_TX_+$, $X\in\textrm{Sm}/k$. 
For instance, $E$ is the algebraic cobordism $T$-spectrum
$MGL$ or motivic sphere spectrum $\cc S_k=(S^0,T,T^2,\ldots)$. Other examples are commutative symmetric ring $T^2$-spectra
$MSL$ and $MSp$~\cite{PW}. The results that use $T^2$-spectra are the same with those proven in this section and which use $T$-spectra.
For brevity, we will deal with $T$-spectra only. 

We can apply Theorem~\ref{tampamore} to the following data:
\begin{itemize}
\item[$\diamond$] $\cc C=\cc M$;
\item[$\diamond$] the canonical map $\sigma:\bb P^{\wedge 1}\to T$, where $\bb P^{\wedge 1}:=(\bb P^1,\infty)\in\cc M$,
given by the framed correspondence $(\{0\},\bb A^1,t)\in\Fr_1(\pt,\pt)$;
\item[$\diamond$] $\cc B=\{X_+\mid X\in\textrm{Sm}/k\}$.
\end{itemize}
Within this notation
the symmetric monoidal spectral category $\corr^{E,\sigma}_*(\cc B)$ of Theorem~\ref{tampamore} will be denoted
by $\cc O^E$ for brevity and each 
$\corr^{E,\sigma}_n(\cc B)(X,Y)=\Hom_{\cc M}(X_+\wedge\bb P^{\wedge n},Y_+\wedge E_n\wedge S^n)$ will be denoted by
$\Fr_n^E(X,Y\otimes S^n)$. Recall from~\cite[Section~3]{GP1} that each simplicial set 
$\Fr_n^E(X,Y\otimes S^n)=\Hom_{\cc M}(X_+\wedge\bb P^{\wedge n},Y_+\wedge E_n\wedge S^n)$ has an explicit 
geometric description due to Voevodsky's Lemma.

Similarly to Definition~\ref{galoissphere} we can consider a spectral category $\cc O^E_{\Delta}$ which is obtained
from $\cc O^E$ by applying the Suslin complex to symmetric spectra of morphisms:
   $$\cc O^E_{\Delta}(X,Y):=(\Fr_0^E(\Delta^\bullet_k\times X,Y),\Fr_1^E(\Delta_k^\bullet\times X,Y\otimes S^1),\ldots).$$

Let $\cc O$ be a spectral category and let $\Mod\cc O$ be the
category of $\cc O$-modules. Recall that the projective stable model
structure on $\Mod\cc O$ is defined as follows (see~\cite{SS1}). The
weak equivalences are the objectwise stable weak equivalences and
fibrations are the objectwise stable projective fibrations. The
stable projective cofibrations are defined by the left lifting
property with respect to all stable projective acyclic fibrations.

Let $\cc Q$ denote the set of elementary distinguished squares in
$\text{Sm}/k$ (see~\cite[3.1.3]{MV})
   \begin{equation*}\label{squareQ}
    \xymatrix{\ar@{}[dr] |{\textrm{$Q$}}U'\ar[r]\ar[d]&X'\ar[d]^\phi\\
              U\ar[r]_\psi&X}
   \end{equation*}
and let $\cc O$ be a spectral category over $\text{Sm}/k$. By $\cc Q_{\cc O}$ denote the
set of squares
   \begin{equation}\label{squareOQ}
    \xymatrix{\ar@{}[dr] |{\textrm{$\cc O Q$}}\cc O(-,U')\ar[r]\ar[d]&\cc O(-,X')\ar[d]^\phi\\
              \cc O(-,U)\ar[r]_\psi&\cc O(-,X)}
   \end{equation}
which are obtained from the squares in $\cc Q$ by taking $X\in
\text{Sm}/k$ to $\cc O(-,X)$. The arrow $\cc O(-,U')\to\cc O(-,X')$
can be factored as a cofibration $\cc O(-,U')\rightarrowtail Cyl$
followed by a simplicial homotopy equivalence $Cyl\to\cc O(-,X')$.
There is a canonical morphism $A_{\cc O Q}:=\cc O(-,U)\bigsqcup_{\cc
O(-,U')} Cyl\to\cc O(-,X)$.

\begin{dfn}[see~\cite{GP12,GP14}]\label{1214}
We say that $\cc O$ is {\it Nisnevich excisive\/} if for every
elementary distinguished square $Q$
the square $\cc O Q$~\eqref{squareOQ} is homotopy pushout in the
Nisnevich local model structure on $Sp_{S^1}^\Sigma(k):=Sp^\Sigma(\cc M,{S^1})$.

The {\it Nisnevich local model structure\/} on $\Mod\cc O$ is the
Bousfield localization of the stable projective model structure with
respect to the family of projective cofibrations
   \begin{equation*}\label{no}
    \cc N_{\cc O}=\{\cyl(A_{\cc O Q}\to\cc O(-,X))\}_{\cc Q_{\cc O}}.
   \end{equation*}
The homotopy category for the Nisnevich local model structure will
be denoted by $\shnis\cc O$. 
\end{dfn}

Suppose $\cc O$ is symmetric monoidal. By a theorem of
Day~\cite{Day} $\Mod\cc O$ is a closed symmetric monoidal category
with smash product $\wedge$ and $\cc O(-,\pt)$ being the monoidal
unit. The smash product is defined as
   \begin{equation}\label{smash}
    M\wedge_{\cc O} N=\int^{\Ob\cc O\otimes\cc O}M(X)\wedge N(Y)\wedge\cc O(-,X\times Y).
   \end{equation}
The internal Hom functor, right adjoint to $-\wedge_{\cc O}M$, is given by
   $$\underline{\Mod}\cc O(M,N)(X):=Sp^\Sigma(M,N(X\times-))=\int_{Y\in\Ob\cc O}\underline{Sp}^\Sigma(M(Y),N(X\times Y)).$$
By~\cite[Corollary~2.7]{DRO} that there is a natural isomorphism
   $$\cc O(-,X)\wedge_{\cc O}\cc O(-,Y)\cong\cc O(-,X\times Y).$$

\begin{thm}[\cite{GP12}]\label{modelmot}
Suppose $\cc O$ is a Nisnevich excisive
spectral category. Then the Nisnevich local model
structure on $\Mod\cc O$ is cellular, proper, spectral and weakly
finitely generated. Moreover, a map of $\cc O$-modules is a weak
equivalence in the Nisnevich local model
structure if and only if it is a weak equivalence in the Nisnevich local
model structure on $Sp_{S^1}^\Sigma(k)$. If
$\cc O$ is a symmetric monoidal spectral category then the
model structure on $\Mod\cc O$ is symmetric monoidal with respect
to the smash product~\eqref{smash} of $\cc O$-modules.
\end{thm}

In our setting we regard spectral categories $\cc O^E$, $\cc O^E_{\Delta}$ as 
symmetric monoidal $\Mod\mathbf 1^E$-categories with the same monoidal product on objects as in $\textrm{Sm}/k$
(see Theorem~\ref{tampamore}), where $\Mod\mathbf 1^E$ is the closed symmetric 
monoidal category of Corollary~\ref{symmspcor}. Denote by $\mathsf{Mod}\cc O^E$ and
$\mathsf{Mod}\cc O^E_\Delta$ the closed symmetric
monoidal categories of 
$\cc O^E$- and $\cc O^E_{\Delta}$-modules in the category $\Mod\mathbf 1^E$ 
(see Corollary~\ref{tampamorecor}).

The Nisnevich local model structure on $\mathsf{Mod}\cc O^E$ and
$\mathsf{Mod}\cc O^E_\Delta$ as well as their homotopy categories $\mathsf{SH}^{\nis}_{S^1}\cc O^E$ and
$\mathsf{SH}^{\nis}_{S^1}\cc O^E_\Delta$ are defined similarly to Definition~\ref{1214}.

Given $X\in\textrm{Sm}/k$ 
and a motivic space $G\in\cc M$, denote by $C_*\Fr^E_n(X_+\wedge G)$ the pointed motivic space
$U\in\textrm{Sm}/k\mapsto\Hom_{\cc M}((U\times\Delta_k^\bullet)_+\wedge\bb P^{\wedge n},X_+\wedge G\wedge E_n)$. 
One has a canonical map $C_*\Fr^E_n(X_+\wedge G)\to C_*\Fr^E_{n+1}(X_+\wedge G)$
defined by the composition
   \begin{multline*}
    \Hom_{\cc M}((U\times\Delta_k^\bullet)_+\wedge\bb P^{\wedge n},X_+\wedge G\wedge E_n)\xrightarrow{-\wedge\sigma}\\
       \Hom_{\cc M}((U\times\Delta_k^\bullet)_+\wedge\bb P^{\wedge n+1},X_+\wedge G\wedge E_n\wedge T)\to
       \uhom_{\cc M}((U\times\Delta_k^\bullet)_+\wedge\bb P^{\wedge n+1},X_+\wedge G\wedge E_{n+1}).
    \end{multline*}
We set,
   $$C_*\Fr^E(X_+\wedge G):=\colim(C_*\Fr^E_0(X_+\wedge G)\to C_*\Fr^E_1(X_+\wedge G)\to\cdots).$$
If we drop $\Delta_k^\bullet$ from the definition of $C_*\Fr^E(X_+\wedge G)$, one gets motivic spaces $\Fr^E(X_+\wedge G)$.

\begin{dfn}\label{sigmafrmotvar}
The {\it symmetric $E$-framed motive\/} of a smooth algebraic variety 
$X\in\textrm{Sm}/k$ is the symmetric $S^1$-spectrum 
   $${M}^\Sigma_{E}(X):=(C_*\Fr^E(X),\uhom_{\cc M}(\bb P^{\wedge 1},C_*\Fr^E(X_+\wedge E_1\wedge S^1)),
       \uhom_{\cc M}(\bb P^{\wedge 2},C_*\Fr^E(X_+\wedge E_2\wedge S^2),\ldots)$$
with structure maps defined similarly to ${M}^\Sigma_{fr}(N)$ of Definition~\ref{sigmafrmot}.
\end{dfn}

\begin{rem}\label{mfrvar}
The $E$-framed motive in the sense of the
preceding definition is a bit different from the $E$-framed motive of
$X$ in the sense of~\cite{GN} defined as
   $${M}_{E}(X):=(C_*\Fr^E(X),C_*\Fr^E(X\wedge S^1)),C_*\Fr^E(X\wedge S^2)),\ldots).$$
\end{rem}

\begin{lem}\label{mfrsymmvar}
If $k$ is a perfect field, then the canonical map of ordinary
$S^1$-spectra $M_{E}(X)\to{M}^\Sigma_{E}(X)$ is a level local
equivalence in positive degrees for any $X\in \textrm{Sm}/k$.
\end{lem}

\begin{proof}
The proof is like that of Lemma~\ref{mfrsymm}. We also use~\cite[Section~7]{GN} here.
\end{proof}

Though ${M}_{E}(X)$ is canonically a symmetric
$S^1$-spectrum in $Sp^\Sigma_{S^1}(k)$, where $\Sigma_n$ acts on each space by permuting
$S^n$, the point is that it is not an $\cc O^E_{\Delta}$-module in
contrast with ${M}^\Sigma_{E}(X)$.

\begin{prop}\label{frmodulevar}
Given a field $k$ and $X\in\mathrm{Sm}/k$, the following
statements are true:
\begin{enumerate}
\item ${M}^\Sigma_{E}(X)$ is an $\cc O^E_{\Delta}$-module;
\item the canonical map
$\alpha:\cc O^E_{\Delta}(-,X)\to M_{E}^\Sigma(X)$ in $\mathsf{Mod}\cc O^E_\Delta$ 
is a sectionwise $\pi_*$-isomorphism (i.e. a stable equivalence of
ordinary spectra) whenever the base field $k$ is perfect.
\end{enumerate}
\end{prop}

\begin{proof}
(1). ${M}^\Sigma_{E}(X)$ is an $\cc O^E_{\Delta}$-module for the same reasons as 
the representable $\cc O^E_{\Delta}(-,X)$ is.

(2). The proof is like that of Proposition~\ref{frmodule}. We also use~\cite[Section~7]{GN} here.
\end{proof}

\begin{thm}\label{spectralmore}
Let $k$ be a perfect field. The commutative spectral category $\cc O^E_{\Delta}$ is Nisnevich excisive
and the Nisnevich local model structure on $\mathsf{Mod}\cc O^E_\Delta$ has all the properties of Theorem~\ref{modelmot}.
The category $\mathsf{SH}^{\nis}_{S^1}\cc O^E_{\Delta}$ is closed symmetric monoidal 
compactly generated triangulated with compact
generators being the symmetric $E$-framed motives $\{M_E^\Sigma(X)\mid X\in\textrm{Sm}/k\}$.
The monoidal product $M_E^\Sigma(X)\wedge^L M_E^\Sigma(X)$ in $\mathsf{SH}^{\nis}_{S^1}\cc O^E_{\Delta}$
is isomorphic to $M_E^\Sigma(X\times Y)$.
\end{thm}

\begin{proof}
$\cc O^E_{\Delta}$ is Nisnevich excisive by~\cite[Section~9]{GN}, Lemma~\ref{mfrsymmvar}
and Proposition~\ref{frmodulevar}. The fact that the Nisnevich local model structure on 
$\mathsf{Mod}\cc O^E_\Delta$ has all the properties of Theorem~\ref{modelmot} follows from the 
fact that $\cc O^E_\Delta$ is Nisnevich excisive symmetric monoidal.

$\mathsf{SH}^{\nis}_{S^1}\cc O^E_{\Delta}$ is closed symmetric monoidal 
compactly generated triangulated with compact
generators being the representable $\cc O^E_{\Delta}$-modules $\{\cc O^E_{\Delta}(-,X)\mid X\in\textrm{Sm}/k\}$.
The isomorphism $M_E^\Sigma(X)\wedge M_E^\Sigma(Y)\cong M_E^\Sigma(X\times Y)$ in
$\mathsf{SH}^{\nis}_{S^1}\cc O^E_{\Delta}$ follows from the isomorphism 
   $$\cc O^E_{\Delta}(-,X\times Y)\cong\cc O^E_{\Delta}(-,X)\wedge\cc O^E_{\Delta}(-,Y)\cong\cc O^E_{\Delta}(-,X)\wedge^L\cc O^E_{\Delta}(-,Y)$$
and Proposition~\ref{frmodulevar}(2). The previous proposition also shows that
compact generators can be given by the symmetric $E$-framed motives $\{M_E^\Sigma(X)\mid X\in\textrm{Sm}/k\}$.
\end{proof}

Let $\bb G_m^{\wedge 1}$ be the mapping cone in $\Delta^{\op}\Fr_0(k)$ associated with $1:\pt\to\bb G_m$.
There is a suspension functor 
 $\Sigma^\infty_{\bb G_m}$ from $\mathsf{Mod}\cc O^E_\Delta$ to $(S^1,\bb G)$-bispectra $Sp_{S^1,\bb G}(k)$:
   $$\Sigma^\infty_{\bb G_m}(\cc X):=(\cc X(\pt),\cc X(\bb G_m^{\wedge 1}),\cc X(\bb G_m^{\wedge 2}),\ldots).$$
Here $\cc X(\bb G_m^{\wedge n}):=\cc X\wedge_{\mathsf{Mod}\cc O^E_\Delta}\cc O^E_\Delta(-,\bb G_m^{\wedge n})$
is regarded as a presheaf of $S^1$-spectra.
Each structure map is induced by the adjunction unit morphism
   $$\cc X(\bb G_m^{\wedge n})\to\uhom_{\mathsf{Mod}\cc O^E_\Delta}(\cc O^E_\Delta(-,\bb G_m^{\wedge 1}),\cc X(\bb G_m^{\wedge n+1})).$$

\begin{cor}\label{spectralmorecor}
Let $k$ be a perfect field. There is
a triangulated equivalence of compactly generated triangulated categories
   $$\mathsf{SH}^{\nis}_{S^1}\cc O^E_{\Delta}\simeq\Mod_{SH(k)}^{\eff}E.$$
\end{cor}

\begin{proof}
$SH(k)=\Ho(Sp_T^\Sigma(k))$ is naturally zigzag equivalent to the category of bispectra
$\mathsf{SH}(k)=\Ho(Sp_{S^1,\bb G}(k))$. Let $\wt{\Mod}_{SH(k)}^{\eff}E$ be the essential image of
$\Mod_{SH(k)}^{\eff}E$ under this zigzag equivalence.
The category $\wt{\Mod}_{SH(k)}^{\eff}E$
is compactly generated by the images of $\{X_+\wedge E\mid X\in\textrm{Sm}/k\}$
in $\wt{\Mod}_{SH(k)}^{\eff}E$. By the proof of~\cite[Theorem~9.13]{GN}
the latter are isomorphic in $\wt{\Mod}_{SH(k)}^{\eff}E$ to motivically fibrant bispectra
   $$M_E^{\mathbb G}(X)_f:=(M_E(X)_f,M_E(X_+\wedge\bb G_m^{\wedge 1})_f,M_E(X_+\wedge\bb G_m^{\wedge 2})_f,\ldots),$$
where ``$f$'' refers to level local
fibrant replacements of motivic $S^1$-spectra. We have a triangulated functor of compactly generated triangulated categories
   $$L\Sigma^\infty_{\bb G_m}:\mathsf{SH}^{\nis}_{S^1}\cc O^E_{\Delta}\to\wt{\Mod}_{SH(k)}^{\eff}E.$$
By Lemma~\ref{mfrsymmvar}, Proposition~\ref{frmodulevar} and Theorem~\ref{spectralmore} 
$M_E^{\mathbb G}(X)_f\cong L\Sigma^\infty_{\bb G_m}(\cc O^E_{\Delta}(-,X))$.
It follows that $L\Sigma^\infty_{\bb G_m}$ takes compact generators to compact generators with isomorphic Hom-sets.
It remains to apply Lemma~\ref{equ}.
\end{proof}

Next, we can stabilize our constructions in the $\bb G_m^{\wedge 1}$-direction as follows. Denote by 
$-\boxtimes\bb G_m^{\wedge 1}$ the endofunctor $\cc X\in\mathsf{Mod}\cc O^E_\Delta\mapsto\cc X(\bb G_m^{\wedge 1})$.
Following Hovey~\cite[Section~8]{H}, we consider the stable model structure on $\bb G_m^{\wedge 1}$-symmetric spectra 
$Sp^\Sigma(\mathsf{Mod}\cc O^E_\Delta,\bb G_m^{\wedge 1})$ (we start with the Nisnevich local stable model structure
on $\mathsf{Mod}\cc O^E_\Delta$). Its homotopy category is denoted by $\mathsf{SH}_{S^1,\bb G_m}\cc O^E_\Delta$.
Given $\cc X\in\mathsf{SH}^{\nis}_{S^1}\cc O^E_\Delta$ we write $\cc X(1)$ to denote $\cc X\boxtimes^L\bb G_m^{\wedge 1}$.

\begin{cor}\label{spectralmorecor1}
Let $k$ be a perfect field. There is
a triangulated equivalence of compactly generated triangulated categories
   $$\mathsf{SH}^{\nis}_{S^1,\bb G_m}\cc O^E_\Delta\simeq\Mod_{SH(k)}E,$$
where $\Mod_{SH(k)}E$ is the category of $E$-modules in $SH(k)$. Moreover,
the functor
   $$\Sigma^\infty_{\bb G_m}:\mathsf{SH}^{\nis}_{S^1}\cc O^E_\Delta
       \to\mathsf{SH}^{\nis}_{S^1,\bb G_m}\cc O^E_\Delta$$
is fully faithful. In particular, $\Hom_{\mathsf{SH}^{\nis}_{S^1}\cc O^E_\Delta}(\cc X,\cc X')\to
\Hom_{\mathsf{SH}^{\nis}_{S^1}\cc O^E_\Delta}(\cc X(1),\cc X'(1))$
is an isomorphism for all $\cc X,\cc X'\in\mathsf{SH}^{\nis}_{S^1}\cc O^E_\Delta$.
\end{cor}

\begin{proof}
The proof is similar to that of Corollary~\ref{spectralmorecor}. We compare compact generators and Hom-sets between them
in both categories. 
\end{proof}

It is worth mentioning that we do not use any motivic equivalences or the $\bb A^1$-relation
in any of our definitions above (similarly to constructions of~\cite{GP5}). All constructions here are genuinely local.
On the other hand, we can do the usual Voevodsky approach~\cite{Voe1} to constructing the triangulated category 
of motives $DM^{\eff}(k)$. We  start with the spectral category $\cc O^E$ and Cech local model
structure on the stable model category of $\cc O^E$-modules $\mathsf{Mod}\cc O^E$.

For each finite Nisnevich cover $\{U_i\to X\}$ we let
$\cc O^E(-,\check{U}_*)$ be the realization of the simplicial module which in dimension $n$ is
$\vee_{i_0,\ldots,i_n}\cc O^E(-,U_{i_0\ldots i_n})$, with the obvious
face and degeneracy maps. Here $U_{i_0\ldots i_n}$ stands for the
smooth scheme $U_{i_0}\times_X\cdots\times_X U_{i_n}$. The reader
should not confuse $\cc O^E(-,\check{U}_*)$ with the realization  $\cc O^E(-,\check
C(U_*))$ of the simplicial module which in dimension $n$ is
$\cc O^E(-,\sqcup_{i_0,\ldots,i_n}U_{i_0\ldots i_n})$.  

\begin{lem}\label{cech}
Each natural map
$\vee_{i_0,\ldots,i_n}\cc O^E(-,U_{i_0\ldots i_n})\to\cc O^E(-,\check C(U_n))$ is a schemewise
stable equivalence of ordinary $S^1$-spectra.
\end{lem}

\begin{proof}
It is enough to show that the natural map 
   \begin{multline*}
    \beta:(\Fr_0^E(X,V),\Fr_1^E(X,V\otimes S^1),\ldots)\vee(\Fr_0^E(X,W),\Fr_1^E(X,W\otimes S^1),\ldots)\to\\
       \to(\Fr_0^E(X,V\sqcup W),\Fr_1^E(X,(V\sqcup W)\otimes S^1),\ldots)
   \end{multline*}
is a stable equivalence of ordinary $S^1$-spectra for any $X,V,W\in\textrm{Sm}/k$.
This is a stable equivalence if and only if $\Theta_{S^1}^\infty(\beta)$ is. The latter map is a stable equivalence if and only if
   \begin{multline*}
    \gamma:(\Fr^E(X,V),\Fr^E(X,V\otimes S^1),\ldots)\vee(\Fr^E(X,W),\Fr^E(X,W\otimes S^1),\ldots)\to\\
       \to(\Fr^E(X,V\sqcup W),\Fr^E(X,(V\sqcup W)\otimes S^1),\ldots)
   \end{multline*}
is a stable equivalence. This is a map of Segal $S^1$-spectra associated to Segal spaces of the form
$K\in\Gamma^{\op}\mapsto\Fr^E(X,V\otimes K)$, hence $\gamma$ is a map of connective spectra.
The Stable Whitehead Theorem~\cite[Proposition~II.6.30]{Sch} implies $\gamma$ is a stable equivalence
if and only if
   \begin{multline*}
    \bb Z(\gamma):(\bb Z\Fr^E(X,V),\bb Z\Fr^E(X,V\otimes S^1),\ldots)\vee(\bb Z\Fr^E(X,W),\bb Z\Fr^E(X,W\otimes S^1),\ldots)\to\\
       \to(\bb Z\Fr^E(X,V\sqcup W),\bb Z\Fr^E(X,(V\sqcup W)\otimes S^1),\ldots)
   \end{multline*}
is a stable equivalence, where $\bb Z\Fr^E(X,V)$ is the reduced free Abelian group of the pointed set $\Fr^E(X,V)$.
Repeating the proof of~\cite[Theorem~1.2]{GNP} word for word, $\bb Z(\gamma)$ is stably equivalent to the map
   \begin{multline*}
     \delta:(\bb ZF^E(X,V),\bb ZF^E(X,V\otimes S^1),\ldots)\vee(\bb ZF^E(X,W),\bb ZF^E(X,W\otimes S^1),\ldots)\to\\
       \to(\bb ZF^E(X,V\sqcup W),\bb ZF^E(X,(V\sqcup W)\otimes S^1),\ldots),
   \end{multline*}
where $\bb ZF^E(X,V)=\colim_n\bb ZF^E_n(X,V)$ with $\bb ZF_n^E(X,V)$ the free Abelian group freely generated by
$E$-framed correspondences with connected support~\cite{GN}. Since $\bb ZF^E(X,V\sqcup W)=\bb ZF^E(X,V)\times\bb ZF^E(X,W)$,
the map $\delta$ equals the stable equivalence of $S^1$-spectra
   \begin{multline*}
     (\bb ZF^E(X,V),\bb ZF^E(X,V\otimes S^1),\ldots)\vee(\bb ZF^E(X,W),\bb ZF^E(X,W\otimes  S^1),\ldots)\to\\
       \to(\bb ZF^E(X,V),\bb ZF^E(X,V\otimes S^1),\ldots)\times(\bb ZF^E(X,W),\bb ZF^E(X,W\otimes S^1),\ldots).
   \end{multline*}
This completes the proof of the lemma.
\end{proof}

The {\it Cech model category $\mathsf{Mod}\cc O^E_{Cech}$\/} associated with
Nisnevich topology is obtained from $\mathsf{Mod}\cc O^E$ by
Bousfield localization with respect to all maps
$\eta:\cc O^E(-,\check{U}_*)\to\cc O^E(-,X)$ running over the set of finite Nisnevich covers. 
It follows from~\cite[Corollary~5.10]{Voe4} (see
also~\cite{DHI}) that $\mathsf{Mod}\cc O^E_{Cech}$ coincides with the
Nisnevich local model category $\mathsf{Mod}\cc O^E_{\nis}$, with
stable weak equivalences defined stalkwise.

We say that a spectral category $\cc O$ is {\it Cech excisive\/} if for any finite
Nisnevich cover $\{U_i\to X\}$ the induced map $\eta:\cc O(-,\check{U}_*)\to\cc O(-,X)$ 
is a local stable weak equivalence.

\begin{thm}\label{spectralcech}
Let $k$ be any field. The commutative spectral category $\cc O^E$ is Cech excisive.
The Cech model structure coincides with 
Nisnevich local model structure on $\mathsf{Mod}\cc O^E$. This model structure 
has all the properties of Theorem~\ref{modelmot}.
The homotopy category $D\cc O^{E,\eff}(k)$ of $\mathsf{Mod}\cc O^E_{Cech}$ is closed symmetric monoidal 
compactly generated triangulated with compact
generators being the representables $\{\cc O^E(-,X)\mid X\in\mathrm{Sm}/k\}$.
The monoidal product $\cc O^E(-,X)\wedge \cc O^E(-,Y)$ in $D\cc O^{E,\eff}(k)$
is isomorphic to $\cc O^E(-,X\times Y)$.
\end{thm}

\begin{proof}
By Lemma~\ref{cech} each map
$\vee_{i_0,\ldots,i_n}\cc O^E(-,U_{i_0\ldots i_n})\to\cc O^E(-,\check C(U_n))$ is a schemewise
stable equivalence, and hence the realization is. The proof of~\cite[Theorem~4.4]{Voe2} shows that the 
map $\cc O^E(-,\check C(U_*))\to\cc O^E(-,X)$ is a level local equivalence. We see that $\eta$
is a local stable weak equivalence. The rest is now straightforward.
\end{proof}

The homotopy category $D\cc O^{E,\eff}(k)$ plays the same role as the derived category
$D(\textrm{Shv}_{tr}^{\nis}(\textrm{Sm}/k))$ of cochain complexes of Nisnevich sheaves with transfers.
Recall from~\cite{Voe1} that Voevodsky's category of motives $DM^{\eff}(k)$ is the localisation of $D(\textrm{Shv}_{tr}^{\nis}(\textrm{Sm}/k))$
with respect to the family $\{\bb Z_{tr}(-,X\times\bb A^1)\to\bb Z_{tr}(-,X)\mid X\in\textrm{Sm}/k\}$. If $k$ is perfect,
$DM^{\eff}(k)$ is equivalent to the full subcategory of $D(\textrm{Shv}_{tr}^{\nis}(\textrm{Sm}/k))$ consisting of
chain complexes with homotopy invariant cohomology sheaves~\cite{Voe1}.
Likewise, localize $\mathsf{Mod}\cc O^E_{Cech}$ with respect to the maps 
$\{\cc O^E(-,X\times\bb A^1)\to\cc O^E(-,X)\mid X\in\textrm{Sm}/k\}$.
Denote by $D\cc O^{E,\eff}_{\mot}(k)$ its homotopy category.

\begin{thm}\label{spectralmorethm}
Let $k$ be a perfect field. The homotopy category $D\cc O^{E,\eff}_{\mot}(k)$ is equivalent to the full triangulated subcategory $DE^{\eff}(k)$
of $D\cc O^{E,\eff}(k)$ consisting of modules with homotopy invariant sheaves of stable homotopy groups.
The inclusion $DE^{\eff}(k)\to D\cc O^{E,\eff}(k)$ has a right adjoint
$C_*$ taking a module $M\in D\cc O^{E,\eff}(k)$ to its Suslin complex $C_*(M)$.

Moreover, there is
a triangulated equivalence of compactly generated triangulated categories
   $$DE^{\eff}(k)\simeq\mathsf{SH}^{\nis}_{S^1}\cc O^E_\Delta.$$
\end{thm}

\begin{proof}
The proof of the first part is like that of~\cite[Theorem~3.5]{GP14}. One also uses here the fact 
that if $k$ is perfect then by~\cite{GP4} (complemented by~\cite{DP} in
characteristic 2 and by~\cite[A.27]{DKO} for finite fields) any
$\bb A^1$-invariant quasi-stable radditive framed presheaf of Abelian
groups $\mathcal F$, the associated Nisnevich sheaf $\mathcal
F_{\nis}$ is strictly $\bb A^1$-invariant.

The equivalence $DE^{\eff}(k)\simeq\mathsf{SH}^{\nis}_{S^1}\cc O^E_\Delta$ follows from the fact that
both categories are compactly generated by symmetric $E$-framed motives with the same Hom-sets
(as usual we use Lemma~\ref{equ} here).
\end{proof}

We call the category $DE^{\eff}(k)$ from the preceding theorem the {\it triangulated category of $E$-framed motives}.
We finish the paper with the following result saying that $DE^{\eff}(k)$ recovers the category of effective $E$-modules in $SH(k)$.

\begin{cor}\label{sledstvie}
Let $k$ be a perfect field. There is a triangulated equivalence of compactly generated triangulated categories
   $$DE^{\eff}(k)\simeq\Mod^{\eff}_{SH(k)}E.$$
\end{cor}

\begin{proof}
This follows from the preceding theorem and Corollary~\ref{spectralmorecor}.
\end{proof}

\end{document}